\documentclass[12pt]{article}

\usepackage{hyperref}

\title{\Large\bf 
Inverse $K$-Chevalley formulas for semi-infinite flag\\[2mm]
manifolds, I: minuscule weights in ADE type
}
\author{%
Takafumi Kouno \\
 \small Department of Mathematics, Tokyo Institute of Technology, \\
 \small 2-12-1 Oh-okayama. Meguro-ku, Tokyo 152-8551, Japan \\
 \small (e-mail: {\tt kouno.t.ab@m.titech.ac.jp}) \\[5mm]
Satoshi Naito \\ 
 \small Department of Mathematics, Tokyo Institute of Technology, \\
 \small 2-12-1 Oh-okayama, Meguro-ku, Tokyo 152-8551, Japan \\
 \small (e-mail: {\tt naito@math.titech.ac.jp}) \\[5mm]
 Daniel Orr \\ 
 \small Department of Mathematics (MC 0123), 460 McBryde Hall, 
        Virginia Tech, \\
 \small 225 Stanger St., Blacksburg, VA 24061, 
 U.\,S.\,A. \\ 
 \small (e-mail: {\tt dorr@vt.edu}) \\[5mm]
Daisuke Sagaki \\ 
 \small Institute of Mathematics, University of Tsukuba, \\
 \small 1-1-1 Tennodai, Tsukuba, Ibaraki 305-8571, Japan \\
 \small (e-mail: {\tt sagaki@math.tsukuba.ac.jp})
}
\date{}

\usepackage{amsmath, amssymb, amsthm, amscd}
\usepackage{bm}
\usepackage{xcolor}

\makeatletter
\renewcommand\section{\@startsection{section}{1}{0pt}
{-3.5ex plus -1ex minus -.2ex}{1.0ex plus .2ex}{\large\bf}}
\renewcommand\subsection{\@startsection{subsection}{1}{0pt}
{2.5ex plus 1ex minus .2ex}{-1em}{\bf}}
\makeatother

\numberwithin{equation}{section}

\theoremstyle{plain}
\newtheorem{thm}{Theorem}[section]
\newtheorem{lem}[thm]{Lemma}
\newtheorem{prop}[thm]{Proposition}

\newtheorem{ithm}{Theorem}

\theoremstyle{definition}
\newtheorem{dfn}[thm]{Definition}

\theoremstyle{remark}

\newtheorem{rem}[thm]{Remark}

\usepackage{color}

\newcommand{\BZ}{\mathbb{Z}}
\newcommand{\BQ}{\mathbb{Q}}
\newcommand{\BC}{\mathbb{C}}

\newcommand{\CO}{\mathcal{O}}

\newcommand{\HH}{\mathbb{H}}
\newcommand{\hHH}{\widehat{\HH}}
\newcommand{\heis}{\mathfrak{H}}
\newcommand{\hHeis}{\widehat{\heis}}

\newcommand{\sG}{\mathsf{G}}
\newcommand{\kqt}{\mathsf{K}}

\newcommand{\Fg}{\mathfrak{g}}
\newcommand{\Fh}{\mathfrak{h}}

\newcommand{\ve}{\varepsilon}
\newcommand{\vpi}{\varpi}

\newcommand{\lng}{w_{\circ}}

\newcommand{\af}{\mathrm{af}}

\newcommand{\wt}{\mathrm{wt}}

\newcommand{\mcr}[1]{\lfloor #1 \rfloor}

\newcommand{\bQG}{\mathbf{Q}_{G}}
\newcommand{\bQGr}{\mathbf{Q}_{G}^{\mathrm{rat}}}
\newcommand{\Kr}{K_{H \times \BC^*}(\bQGr)}
\newcommand{\Walk}{\mathbf{W}}
\newcommand{\WalkStd}{\mathbf{W}_{w}^{\vec{\eta}}}
\newcommand{\QWalk}{\mathbf{QW}}
\newcommand{\DQWalk}{\widetilde{\QWalk}}
\newcommand{\Mat}{\mathrm{Mat}}

\newcommand{\QBG}{\mathrm{QBG}}

\newcommand{\Inv}{\mathrm{Inv}}

\newcommand{\bra}[1]{[\![#1]\!]}
\newcommand{\pra}[1]{(\!(#1)\!)}
\newcommand{\pair}[2]{\langle #1, #2 \rangle}

\newcommand{\tg}{\widetilde{g}}
\newcommand{\tX}{\widetilde{X}}

\newcommand{\J}{J}
\newcommand{\WJ}{W^{\J}}

\newcommand{\DJp}{\Delta^{+} \setminus \Delta_{\J}^{+}}

\newcommand{\Ht}{\mathrm{ht}}

\newenvironment{enu}{%
 \begin{enumerate}%
}{\end{enumerate}}

\allowdisplaybreaks[4]

\begin{document}

\maketitle


\begin{abstract}
We prove an explicit inverse Chevalley formula in the equivariant $K$-theory of semi-infinite flag manifolds of simply-laced type. By an inverse Chevalley formula, we mean a formula for the product of an equivariant scalar with a Schubert class, expressed as a $\BZ[q^{\pm 1}]$-linear combination of Schubert classes twisted by equivariant line bundles. Our formula applies to arbitrary Schubert classes in semi-infinite flag manifolds of simply-laced type and equivariant scalars $e^{\lambda}$, where $\lambda$ is an arbitrary minuscule weight. By a result of Stembridge, our formula completely determines the inverse Chevalley formula for arbitrary weights in simply-laced type, except for type $E_8$. The combinatorics of our formula is governed by the quantum Bruhat graph, and the proof is based on a limit from the double affine Hecke algebra. As such, our formula also provides an explicit determination of all nonsymmetric $q$-Toda operators for minuscule weights in ADE type.
\end{abstract}


\section{Introduction.} 
\label{sec:intro}

Let $\bQGr$ be the semi-infinite flag manifold. This is a 
reduced ind-scheme whose set of 
$\BC$-valued points is $G(\BC\pra{z})/(H(\BC) \cdot N(\BC\pra{z}))$ 
(see \cite{Kat2} for details), where $G$ is 
a simply-connected simple algebraic group over $\BC$, $B=HN\subset G$ is a Borel subgroup, $H$ is a maximal torus, and $N$ is the unipotent radical of $B$.
For each affine Weyl group element $x\in W_{\af} = W \ltimes Q^{\vee}$, with $W = \langle s_i \mid i \in I \rangle$ the (finite) Weyl group and $Q^{\vee} = \bigoplus_{i \in I} \BZ \alpha_i^{\vee}$ 
the coroot lattice of $G$, one has a semi-infinite Schubert variety $\bQG(x)\subset\bQGr$, which is infinite-dimensional and is given as an orbit closure for the Iwahori subgroup $\mathbf{I}\subset G(\BC\bra{z})$.
We distinguish the semi-infinite Schubert variety 
$\bQG := \bQG(e) \subset \bQGr$ associated to the identity element $e$ of 
the affine Weyl group and also call $\bQG$ the semi-infinite flag manifold. 

Our main object of study is the equivariant $K$-group $K_{H \times \BC^{\ast}}(\bQG)$ (and that of $\bQGr$ denoted by $K_{H \times \BC^{\ast}}(\bQGr)$), which is a variant of the $K$-group $K_{\tilde{\mathbf{I}}}^{\prime}(\bQG)$ introduced recently in \cite{KNS}. Our $K$-group is a module over the equivariant scalar ring $\BZ[q^{\pm 1}][P]$, where $P$ is the weight lattice of $G$, $\BZ[P]=\BZ[e^{\mu} : \mu\in P]$ is the character ring of $H$, and $q\in R(\BC^*)$ is the character of  loop rotation. As such, the $K$-group $K_{H \times \BC^{\ast}}(\bQG)$ is a $\BZ[q^{\pm 1}][P]$-submodule of an extension of scalars of the equivariant $K$-group $K_{\tilde{\mathbf{I}}}^{\prime}(\bQG)$ of \cite{KNS} with respect to the Iwahori subgroup $\mathbf{I}$ and loop rotation.

A fundamental result of \cite{KNS} is the combinatorial Chevalley formula for dominant weights 
in the $K$-group $K_{\tilde{\mathbf{I}}}^{\prime}(\bQG)$ (and hence in $K_{H \times \BC^{\ast}}(\bQG)$). This formula describes, 
in terms of semi-infinite Lakshmibai-Seshadri paths, 
the tensor product of the class of the line bundle $[\CO_{\bQG}(\lambda)]$ 
associated to a dominant weight $\lambda\in P^+$ with the class of 
the structure sheaf $[\CO_{\bQG(x)}]$ of 
a semi-infinite Schubert variety $\bQG(x) \subset \bQG$ 
for $x = w t_{\xi} \in W_{\af}^{\geq 0} := W \times Q^{\vee,+} \subset W_{\af}$, 
where $Q^{\vee,+} := \sum_{i \in I} \BZ_{\geq 0} \alpha_{i}^{\vee} \subset Q^{\vee}$.
This was followed up in \cite{NOS} by another combinatorial Chevalley formula in $K_{H\times\BC^*}(\bQG)$, giving the tensor product of a Schubert class with an {\em antidominant} line bundle. The two Chevalley formulas---dominant \cite{KNS} and antidominant \cite{NOS}---were unified in \cite{LNS}, giving the general Chevalley formula in $K_{H\times\BC^*}(\bQG)$ for arbitrary weights $\lambda\in P$.

The Chevalley formulas of \cite{KNS,NOS,LNS} thus provide the complete analog for semi-infinite flag manifolds of their previously well-understood $K$-theory counterparts for the standard Kac-Moody flag varieties \cite{PR99,LiSe03,GR04,LP,LeSh14}. In all such formulas, the objective is to expand the tensor product of a Schubert class with an equivariant line bundles, as a linear combination of Schubert classes with equivariant scalar coefficients. In the case of $\bQG$, this takes the form:
\begin{align}\label{E:chev}
[\CO_{\bQG(x)}(\lambda)] &= \sum_{\substack{y\in W_\af^{\ge 0}\\\mu\in P}} c_{x,y}^{\lambda,\mu} \cdot e^\mu \cdot [\CO_{\bQG(y)}]
\end{align}
where $x\in W_\af^{\ge 0}$, $\lambda\in P$, and $c_{x,y}^{\lambda,\mu}\in\BZ\bra{q^{-1}}$. The (generally infinite) sum on the right-hand side of \eqref{E:chev} satisfies a notion of convergence introduced in \cite{KNS}.

In this paper, we shall study the inverse expansion in $K_{H\times\BC^*}(\bQG)$:
\begin{align}\label{E:inv-chev}
e^\lambda\cdot [\CO_{\bQG(x)}] &= \sum_{\substack{y\in W_\af^{\ge 0}\\\mu\in P}} d_{x,y}^{\lambda,\mu}\cdot [\CO_{\bQG(y)}(\mu)]
\end{align}
for $x\in W_\af^{\ge 0}$ and $\lambda\in P$. In contrast to \eqref{E:chev}, the expansion \eqref{E:inv-chev} exhibits {\em finiteness}, as established in \cite{Orr} for simply-laced $G$, namely: (i) the right-hand side of \eqref{E:inv-chev} is always a finite sum, and (ii) $d_{x,y}^{\lambda,\mu}\in\BZ[q^{\pm 1}]$, for arbitrary $\lambda,\mu\in P$ and $x,y\in W_\af^{\ge 0}$. (These properties are expected to hold for arbitrary $G$.)

We call any formula for the right-hand side of \eqref{E:inv-chev} an {\em inverse Chevalley formula} for $\lambda\in P$ in $K_{H\times\BC^*}(\bQG)$. The analogous expansion for finite-dimensional flag manifolds $G/B$ was studied by Mathieu \cite{M00} in the context of filtrations of $B$-modules. In fact, one can show that the truncation of \eqref{E:inv-chev} to $x,y\in W$ recovers the expansion of \cite[p. 239]{M00} in $K_H(G/B)$, with $[\CO_{\bQG(w)}]$ for $w\in W$ corresponding to the class $[\CO_{X(w)}]$ given by the structure sheaf of the Schubert variety
\begin{align}\label{E:fin-schub}
X(w)=\overline{Bww_\circ B/B}\subset G/B
\end{align}
where $w_\circ$ is the longest element of $W$.
Thus \eqref{E:inv-chev} incorporates the necessary ``corrections'' in $K_{H\times\BC^*}(\bQG)$ to the classical inverse Chevalley formulas in $K_H(G/B)$.

In the case of $K_H(G/B)$, there is a simple transformation to pass between ordinary and inverse Chevalley formulas (as explained in, e.g., \cite[p. 239]{M00}). The lack of such a transformation for $K_{H\times\BC^*}(\bQG)$, especially in light of the finiteness of \eqref{E:inv-chev}, justifies the independent study of inverse Chevalley formulas in the semi-infinite setting. 

The purpose of this paper is to prove a completely explicit, combinatorial inverse Chevalley formula in the equivariant $K$-group $K_{H \times \BC^*}(\bQG)$ in the case of simply-laced group $G$ and a minuscule weight $\lambda\in P$. Before stating our results more precisely, let us discuss further motivations for this work.


\subsection{nil-DAHA and Heisenberg actions on $\Kr$.}
One approach to understanding formulas \eqref{E:chev} and \eqref{E:inv-chev} is that they relate two actions of the group algebra $\BZ[P]$ on $\Kr$, one given by the tensor product with equivariant line bundles (the left-hand side of \eqref{E:chev}) and the other by equivariant scalar multiplication (the left-hand side of \eqref{E:inv-chev}). These actions of $\BZ[P]$ extend to that of two distinct algebras on $\Kr$, the nil double affine Hecke algebra (nil-DAHA) and a $q$-Heisenberg algebra.

Multiplication by equivariant scalars extends to a left action of the nil-DAHA $\HH_0$ on $\Kr$, in a way that is conceptually similar to the action of nil-Hecke algebras on the equivariant $K$-theory of Kac-Moody flag varieties \cite{KK90}. On the right, however, instead of a nil-Hecke algebra, one has an action of a $q$-Heisenberg algebra $\heis$. This is generated by tensor products with equivariant line bundles $[\CO(\lambda)] \ (\lambda\in P)$ and translations $Q^\vee \cong H(\BC\pra{z})/H(\BC\bra{z})$. The two actions commute, making $\Kr$ an $(\HH_0,\heis)$-bimodule. 

This bimodule structure is a fundamental tool in the study of $\Kr$, as standard methods such as localization are not available. Furthermore, as explained in \cite{Orr}, it gives a geometric realization of the nonsymmetric $q$-Toda system introduced in \cite{CO}; see \cite{GL03,BF1,BF2} for related results on the usual $q$-Toda system and \cite{Ko,KoZ} for its $(\mathsf{q},\mathsf{t})$-extension given by Macdonald difference operators in type $A$. For us, the fact that the $\HH_0$-action on $\Kr$ includes the operators of multiplication by equivariant scalars is key. We use the limit construction \cite{Orr} of the $\HH_0$-action on $\Kr$ to find and prove our inverse Chevalley formula in $K_{H\times\BC^*}(\bQG)$, which is given by Theorem~\ref{ithm1} below. 


\subsection{Quantum $K$-theory of $G/B$.}
In \cite{Kat1,Kat3}, Kato has established a $\BZ[P]$-module isomorphism---with $\BZ[P]$ acting by  equivariant scalars---from the (completed) $H$-equivariant quantum $K$-group $QK_{H}(G/B) := K_{H}(G/B) \otimes \BZ[\![Q^{\vee,+}]\!]$ 
of the finite-dimensional flag manifold $G/B$ onto 
the $H$-equivariant $K$-group $K_{H}(\bQG)$ obtained by specializing $q=1$ in $K_{H\times\BC^*}(\bQG)$.
(Here $\BZ[\![Q^{\vee,+}]\!]$ is the ring of formal power series 
in the (Novikov) variables $Q_{i}$, $i \in I$.) Kato's isomorphism respects Schubert classes and intertwines the quantum multiplication in $QK_{H}(G/B)$ with the tensor product by line bundles in $K_H(\bQG)$. Thus it provides a means to transport formulas from $K_H(\bQG)$ to $QK_H(G/B)$. In the sequel \cite{KNOS2} to this paper, we will use Kato's isomorphism to derive a corresponding inverse Chevalley formula in $QK_{H}(G/B)$.

\subsection{Our results.}
Let us now explain our results in more detail. Recall that a weight $\lambda\in P$ is called minuscule if $\langle \lambda,\alpha^\vee\rangle\in\{0,\pm 1\}$ for all $\alpha\in\Delta$. Nonzero minuscule weights exist in all types except $E_8,F_4$, and $G_2$.

Our results explicitly describe the inverse Chevalley formula in $K_{H\times\BC^*}(\bQG)$ for arbitrary minuscule weights $\lambda\in P$ in the case when $G$ is simply-laced. By iteration, our formulas completely determine the inverse Chevalley rule for arbitrary weights in ADE type (except in type $E_8$). Indeed, a result of Stembridge \cite{Stem} (as stated in \cite[Theorem 2.1]{Len}) asserts that, in all types except $E_8, F_4$, and $G_2$, the minuscule weights form a set of generators for the weight lattice. (The full version of Stembridge's result, which holds in arbitrary type, requires quasi-minuscule weights. We plan to take up the study of our constructions in the case of quasi-minuscule weights elsewhere.)

Any minuscule weight belongs to the Weyl group orbit of a dominant minuscule weight, i.e., a minuscule fundamental weight. Suppose $\varpi_k\in P^+$ is a minuscule fundamental weight. Set $J=I\setminus\{k\}$, and consider the parabolic subgroup $W_J=\langle s_j \mid j\in J\rangle$, which is the stabilizer of $\varpi_k$. Let $W^J$ be the set of minimal coset representatives for $W/W_J$. Finally, let $\lambda=x\varpi_k\in P$ be an arbitrary minuscule weight, where $x\in W^J$.

\subsubsection{Algebraic formula.}

For $G$ simply-laced, $\lambda$ as above, and any $w\in W$, our first main result gives an algebraic expression for the product $e^\lambda\cdot[\CO_{\bQG(w)}]\in K_{H\times\BC^*}(\bQG)$, in terms of the right $q$-Heisenberg action on $K_{H\times\BC^*}(\bQGr)$. This is given explicitly as a sum over a set $\QWalk_{\lambda,w}$ of walks $(w_1,\dotsc,w_n)$ in the quantum Bruhat graph $\QBG(W)$ \cite{BFP}, beginning at $w_0=w$ and with steps prescribed by a set of positive roots determined by $\lambda$. (Here $n$ is the length of the minimal representative of the coset $\lng W_J$.)

\begin{ithm}[= Theorem~\ref{T:K-w-arb}]\label{ithm1}
Assume that $G$ is of type $ADE$, but not of type $E_8$. For any minuscule weight $\lambda=x\varpi_k\in P$, where $x\in W^J$, and any $w\in W$, we have:
\begin{align}\label{E:ithm1}
e^\lambda\cdot[\CO_{\bQG(w)}] &= \sum_{\mathbf{w}=(w_1,\dotsc,w_n)\in\QWalk_{\lambda,w}}[\CO_{\bQG(w_n)}]\cdot  \tg^+_{\mathbf{w}} \, X^{-\lng w_l^{-1}\lambda}\, \tg^-_{\mathbf{w}},
\end{align}
where $l=\ell(x)$ and $\tg^+_{\mathbf{w}} \, X^{-\lng w_l^{-1}\nu}\, \tg^-_{\mathbf{w}}$ is an element of the $q$-Heisenberg algebra $\heis$ given explicitly by \eqref{E:tg-} and \eqref{E:tg+}.
\end{ithm}

Acting by translations from the $q$-Heisenberg algebra, one immediately obtains a corresponding formula for $e^\lambda\cdot[\CO_{\bQG(wt_\xi)}]\in K_{H\times\BC^*}(\bQGr)$, for any $\xi\in Q^\vee$.

In order to prove Theorem~\ref{ithm1}, we apply the main result of \cite{Orr}. The main step in our proof (see Theorem~\ref{T:Y-lim}) is the intricate computation of a limit, as $\mathsf{t}\to 0$, from the polynomial representation of the double affine Hecke algebra (DAHA). Thus one should also regard Theorem~\ref{ithm1} as an explicit determination of all nonsymmetric $q$-Toda operators, in the sense of \cite{CO,Orr}, for minuscule weights in ADE type.

\subsubsection{Combinatorial formula.}

Our second main result expresses the same product $e^\lambda\cdot[\CO_{\bQG(w)}]$ combinatorially. To achieve this, we enhance the set of quantum walks $\QWalk_{\lambda,w}$ to a set $\DQWalk_{\lambda,w}$ of {\em decorated quantum walks}. Roughly speaking, a decorated quantum walk $(\mathbf{w},\mathbf{b})\in \DQWalk_{\lambda,w}$ consists of a quantum walk $\mathbf{w}=(w_1,\dotsc,w_n)\in\QWalk_{\lambda,w}$, together with a decoration $\mathbf{b}: S(\mathbf{w})\to\{0,1\}$. The latter is a $\{0,1\}$-valued function on an explicit subset $S(\mathbf{w})\subset \{t : w_t=w_{t-1}\}$ of the stationary steps in the walk $\mathbf{w}$. Each $(\mathbf{w},\mathbf{b})\in \DQWalk_{\lambda,w}$ carries a sign $(-1)^{(\mathbf{w},\mathbf{b})}\in\{\pm 1\}$, a weight $\wt(\mathbf{w},\mathbf{b})\in Q^\vee$, and a degree $\deg(\mathbf{w},\mathbf{b})\in\BZ$. For details, see \S\ref{SS:DQW}.

Our combinatorial inverse Chevalley formula reads as follows:

\begin{ithm}[= Theorem~\ref{T:K-w-arb-dec}]\label{ithm2}
Assume that $G$ is of type $ADE$, but not of type $E_8$. For any minuscule weight $\lambda=x\varpi_k\in P$, where $x\in W^J$, and any $w\in W$, we have
\begin{align}\label{E:ithm2}
&e^{\lambda}\cdot[\CO_{\bQG(w)}]\\
&\quad = \sum_{(\mathbf{w},\mathbf{b})\in\DQWalk_{\lambda,w}}(-1)^{(\mathbf{w},\mathbf{b})}q^{\deg(\mathbf{w},\mathbf{b})}\cdot[\CO_{\bQG(w_nt_{-w_\circ(\wt(\mathbf{w},\mathbf{b}))})}(-w_\circ w_l^{-1}\lambda+\wt(\mathbf{w},\mathbf{b}))].\notag
\end{align}
\end{ithm}

Theorem~\ref{ithm2} is obtained as an immediate consequence of Theorem~\ref{ithm1}, by fully expanding the right-hand side of \eqref{E:ithm1} in the $q$-Heisenberg algebra. Our combinatorial framework is designed to record the terms in this expansion. We find it satisfying that the DAHA-based limit used to prove \eqref{E:ithm1} automatically manufactures the combinatorics of decorated quantum walks necessary for both formulas.
We mention that \eqref{E:ithm1} gives the extension to $K_{H\times\BC^*}(\bQG)$ of Lenart's rule \cite[Theorem 3.1]{L}, which holds in $K(SL(n+1)/B)$, for multiplying a Grothendieck polynomial by a variable. 

In Appendix~\ref{A:A}, we work out some further details in the type $A$ case, including an important special case of \eqref{E:ithm2} in \eqref{E:K-w0}. We also explain, following \cite[\S5.1]{Orr}, how to obtain the usual $q$-Toda difference operator by symmetrizing Theorem~\ref{ithm1}.

\subsection{Sequel.} In \cite{KNOS2}, we will
establish a different, equivalent version of our inverse Chevalley formula \eqref{E:ithm2} in terms of paths (instead of walks) in the quantum Bruhat graph. By means of this alternate formula, we will give a separate and logically independent proof of Theorem \ref{ithm2} in type $A$, based on the Chevalley formulas of \cite{KNS,NOS} and an equivalent set of character identities for Demazure submodules of level-zero extremal weight modules. Finally, in \cite{KNOS2}, we will use Theorem~\ref{ithm2} to derive a corresponding inverse Chevalley formula in the quantum $K$-ring $QK_H(G/B)$, by means of Kato's isomorphism.


\subsection*{Acknowledgements.}
The authors would like to thank Cristian Lenart and Mark Shimozono for helpful discussions.
T.K. was supported in part by Grant-in-Aid for JSPS Fellows 20J12058.
S.N. was supported in part by JSPS Grant-in-Aid for Scientific Research (B) 16H03920.
D.O. was supported in part by a Collaboration Grant for Mathematicians from the Simons Foundation.
D.S. was supported in part by JSPS Grant-in-Aid for Scientific Research (C) 19K03415.


\section{Basic notation.}\label{sec:pre}


\subsection{Root system.}
Let $G$ be a simply-connected simple algebraic group over $\BC$. As in the introduction, we fix a maximal torus and Borel subgroup $H\subset B\subset G$. Set $\Fg := \mathrm{Lie}(G)$ and $\Fh := \mathrm{Lie}(H)$.
We denote by $\pair{\cdot}{\cdot} : \Fh^{\ast} \times \Fh \rightarrow \BC$ 
the canonical pairing, where $\Fh^{\ast} = \mathrm{Hom}_{\BC}(\Fh, \BC)$. 

Let $\Delta \subset \Fh^{\ast}$ be the root system of $\Fg$, $\Delta^{+} \subset \Delta$ the positive roots (with respect to $B$), 
and $\{ \alpha_{i} \}_{i \in I} \subset \Delta^{+}$ the set of simple roots. 
We denote by $\alpha^{\vee} \in \Fh$ the coroot corresponding to $\alpha \in \Delta$. 
Also, we denote by $\theta \in \Delta^+$ the highest root of $\Delta$, and 
we set $\rho := (1/2) \sum_{\alpha \in \Delta^{+}} \alpha$. 
The root lattice $Q$ and the coroot lattice $Q^{\vee}$ of $\Fg$ are 
$Q := \sum_{i \in I} \BZ \alpha_{i}$ and $Q^{\vee} := \sum_{i \in I} \BZ \alpha_{i}^{\vee}$. 

For $i \in I$, let $\vpi_{i} \in \Fh^{\ast}$ be the fundamental weight determined by
$\pair{\vpi_{i}}{\alpha_{j}^{\vee}} = \delta_{i, j}$ for all $j \in I$, 
where $\delta_{i, j}$ denotes the Kronecker delta.
The weight lattice $P$ of $\Fg$ is defined by $P := \sum_{i \in I} \BZ \vpi_{i}$. 
We denote by $\BZ[P]$ the group algebra of $P$, that is, 
the associative algebra generated by formal elements $\{ e^{\lambda} \mid \lambda \in P\}$, 
where the product is defined by $e^{\lambda} e^{\mu} := e^{\lambda + \mu}$ for $\lambda, \mu \in P$. 

A reflection $s_{\alpha} \in GL(\Fh^{\ast})$, $\alpha \in \Delta$, 
is defined by $s_{\alpha}(\lambda) := \lambda - \pair{\lambda}{\alpha^{\vee}} \alpha$ 
for $\lambda \in \Fh^{\ast}$. We write $s_{i} := s_{\alpha_{i}}$ for $i \in I$. 
Then the Weyl group $W := \langle s_{i} \mid i \in I \rangle$ of $\Fg$ is the subgroup of $GL(\Fh^{\ast})$ generated by $\{ s_{i} \}_{i \in I}$.
We denote by $\ell(w)$ the length of $w\in W$ with respect to $\{ s_{i} \}_{i \in I}$.

\subsection{Quantum Bruhat graph.}\label{S:QBG}
The \emph{quantum Bruhat graph} $\QBG(W)$ (cf. \cite[Definition~6.1]{BFP})
is the $\Delta^+$-labeled directed graph whose vertices are the elements of $W$ 
and whose edges are of the following form: 
$x \xrightarrow{\alpha} y$, with $x, y \in W$ and $\alpha \in \Delta^+$, 
such that $y = x s_{\alpha}$ and either of the following holds: 
(B) $\ell(y) = \ell(x) + 1$, (Q) $\ell(y) = \ell(x) - 2\pair{\rho}{\alpha^\vee} + 1$. 
An edge satisfying (B) (resp. (Q)) is called a \emph{Bruhat edge} (resp. a \emph{quantum edge}).


\subsection{Affine root system.}
Let $\Fg_{\af} := (\Fg \otimes \BC[z, z^{-1}]) \oplus \BC c \oplus \BC d$ 
be the (untwisted) affine Lie algebra over $\BC$ associated to $\Fg$, 
where $c$ is the canonical central element and $d$ is the degree operator. 
Then $\Fh_{\af} := \Fh \oplus \BC c \oplus \BC d$ is the Cartan subalgebra of $\Fg_{\af}$. 
We denote by $\pair{\cdot}{\cdot} : \Fh_{\af}^{\ast} \times \Fh_{\af} \rightarrow \BC$ the canonical pairing.
Regarding $\lambda \in \Fh^{\ast}$ as 
$\lambda \in \Fh_{\af}^{\ast} = \mathrm{Hom}_{\BC}(\Fh_{\af}, \BC)$ 
by setting $\pair{\lambda}{c} = \pair{\lambda}{d} = 0$, 
we have $\Fh^{\ast} \subset \Fh_{\af}^{\ast}$. 
In this identification, we see that the canonical pairing $\pair{\cdot}{\cdot}$ 
on $\Fh_{\af}^{\ast} \times \Fh_{\af}$ 
extends that on $\Fh^{\ast} \times \Fh$. 

Let us consider the root system of $\Fg_{\af}$. 
We define $\delta$ to be the unique element of $\Fh_{\af}^{\ast}$ which satisfies
$\pair{\delta}{h} = 0$ for all $h \in \Fh$, $\pair{\delta}{c} = 0$, and $\pair{\delta}{d} = 1$.  
We set $\alpha_{0} := -\theta + \delta \in \Fh_{\af}^{\ast}$. 
Then the root system $\Delta_{\af}$ of $\Fg_{\af}$ 
has simple roots $\{ \alpha_{i} \}_{i \in I_{\af}}$, where $I_{\af} := I \sqcup \{ 0 \}$.

For each $\alpha \in \Delta_{\af}$, we have a reflection $s_{\alpha} \in GL(\Fh_{\af})$, defined as for $\Fg$. 
Note that for $\alpha \in \Delta \subset \Delta_{\af}$, 
the restriction of a reflection $s_{\alpha}$ defined on $\Fh_{\af}$ to $\Fh$ 
coincides with a reflection $s_{\alpha}$ defined on $\Fh$. 
Set $s_{i} := s_{\alpha_{i}}$ for $i \in I_{\af}$. 
Then, the Weyl group of $\Fg_{\af}$ (called the affine Weyl group) $W_{\af}$ is defined 
to be the subgroup of $GL(\Fh_{\af})$ generated by $\{ s_{i} \}_{i \in I_{\af}}$, 
namely, $W_{\af} = \langle s_{i} \mid i \in I_{\af} \rangle$. 
In \cite[\S 6.5]{Kac},
it is shown that $W_{\af} \simeq W \ltimes \{ t_{\alpha^{\vee}} \mid \alpha^{\vee} \in Q^{\vee} \} \simeq W \ltimes Q^{\vee}$, 
where $t_{\alpha^{\vee}}$ is the translation element corresponding to $\alpha^{\vee} \in Q^{\vee}$.

\subsection{Simply-laced assumption.}
While the definitions above are completely general, we will in fact assume throughout that $G$ is {\em simply-laced}. As a result, we almost always identify $\Fh^\ast$ with $\Fh$, roots with coroots, and so on. We do this by means of the non-degenerate $W$-invariant symmetric bilinear form $(\cdot,\cdot) : \Fh^* \times \Fh^* \to \BC$, normalized so that $(\alpha,\alpha)=2$ for all $\alpha\in\Delta$. We write $|\lambda|^2=(\lambda,\lambda)$ for $\lambda\in\Fh^*$.

\subsection{Extended affine Weyl group.} 
Let $W_{\mathrm{ex}}=W\ltimes P$ be the extended affine Weyl group, which we let act on $P\oplus\BZ\frac{\delta}{e}$ by the level-zero action: 
\begin{align}
wt_\mu (\lambda)=w(\lambda)-(\mu,\lambda)\delta,\qquad wt_\mu(\delta)=\delta.
\end{align}
Here we choose $e$ to be the smallest positive integer such that $e\cdot(P,P)\subset\mathbb{Z}$.

We also define the group $\Pi=P/Q$, which we realize as the subgroup of length zero elements in $W_{\mathrm{ex}}$.

\subsection{Parameters.}
Let us introduce a parameter $q^{1/e}$ such that $(q^{1/e})^e=q$, where $q\in R(\BC^*)$ is the equivariant parameter corresponding to loop rotation on $\bQGr$. Below we will also use a related, but distinct parameter $\mathsf{q}^{1/e}$, as well as parameters $\mathsf{t},\mathsf{v}$ such that $\mathsf{t}=\mathsf{v}^2$. Our base field for DAHA constructions will be $\kqt=\BQ(\mathsf{q}^{1/e},\mathsf{v})$.

\subsection{Matrices.}\label{S:mat}
For any ring $R$ with $1$, let $\Mat_W(R)$ denote the $R$-algebra of $W\times W$ matrices with entries in $R$. Let $\{e_w\}_{w\in W}$ be the standard basis of the free $R$-module $R^{|W|}$. For $A\in\Mat_W(R)$, we denote by $A_{w,\bullet}=\sum_{v\in W}A_{w,v}e_v$ the row of the matrix $A$ indexed by $w\in W$.


\section{Inverse Chevalley formula via DAHA.}\label{S:DAHA}

In this section we use the methods of \cite{Orr}, based on the $(\HH_0,\heis)$-bimodule structure of $\Kr$, to prove the inverse Chevalley formulas \eqref{E:ithm1} and \eqref{E:ithm2}.

\subsection{$K$-groups.}

Let $K_{\tilde{\mathbf{I}}}(\bQGr)$ be the equivariant $K$-group of $\bQGr$ introduced in \cite[\S 6]{KNS}, where $\tilde{\mathbf{I}}=\mathbf{I}\rtimes\BC^*$ is semi-direct product of the Iwahori subgroup $\mathbf{I}\subset G(\BC\bra{z})$ and loop rotation $\BC^*$. Correspondingly, $K_{\tilde{\mathbf{I}}}(\bQGr)$ is a module over $\BZ[P]\pra{q^{-1}}$, which acts by equivariant scalar multiplication. 

One has the following classes in $K_{\tilde{\mathbf{I}}}(\bQGr)$, for each $x\in W_\af$ and $\lambda\in P$:
\begin{itemize}
\item Schubert classes $[\CO_{\bQG(x)}]$, 
\item equivariant line bundle classes $[\CO(\lambda)]$,
\item classes $[\CO_{\bQG(x)}(\lambda)]$ corresponding to the tensor product sheaves $\CO_{\bQG(x)}\otimes\CO(\lambda)$.
\end{itemize}
We follow the conventions of \cite{KNS} for indexing equivariant line bundles and Schubert varieties in $\bQGr$. 

\begin{dfn}[$(H\times\BC^*)$-equivariant $K$-groups of $\bQGr$ and $\bQG$]
Let $K_{H\times\BC^*}(\bQGr)$ be the $\BZ[q^{\pm 1}][P]$-submodule of $K_{\tilde{\mathbf{I}}}(\bQGr)$ consisting of all convergent infinite $\BZ[q^{\pm 1}][P]$-linear combinations of Schubert classes $[\CO_{\bQG(x)}]$ for $x\in W_\af$, where convergence holds in the sense of \cite[Proposition 5.8]{KNS}. 

Similarly, we define $K_{H\times\BC^*}(\bQG)$ to be the $\BZ[q^{\pm 1}][P]$-submodule of $K_{\tilde{\mathbf{I}}}(\bQGr)$ consisting of all convergent infinite $\BZ[q^{\pm 1}][P]$-linear combinations of Schubert classes $[\CO_{\bQG(x)}]$ for $x\in W_\af^{\ge 0}$. 
\end{dfn}

The classes $\{[\CO_{\bQG(x)}]\}_{x\in W_\af}$ satisfy a notion of topological linear independence in $K_{\tilde{\mathbf{I}}}(\bQGr)$ given by \cite[Proposition 5.8]{KNS}; thus they form a topological $\BZ[q^{\pm 1}][P]$-basis of $K_{H\times\BC^*}(\bQGr)$. Also, one has $[\CO_{\bQG(x)}(\lambda)]\in K_{H\times\BC^*}(\bQGr)$ for any $x\in W_\af$ and $\lambda\in P$, thanks to the Chevalley formula of \cite{KNS}. Similar (in fact, equivalent) assertions hold for $K_{H\times\BC^*}(\bQG)$.

\begin{dfn}
Define $\mathbb{K}\subset \Kr$ to be the $\BZ[q^{\pm 1}]$-submodule consisting of all finite $\BZ[q^{\pm 1}]$-linear combinations of the classes $\{[\CO_{\bQG(x)}(\lambda)]\}_{x\in W_\af,\lambda\in P}$.
\end{dfn}

By definition, $\mathbb{K}$ is only a $\BZ[q^{\pm 1}]$-submodule of $\Kr$. But, as shown in \cite[Theorem 5.1]{Orr}, $\mathbb{K}$ is indeed a $\BZ[q^{\pm 1}][P]$-submodule of $\Kr$. We note that the classes $\{[\CO_{\bQG(x)}(\lambda)]\}_{x\in W_\af,\lambda\in P}$ are linearly independent over $\BZ[q^{\pm 1}]$, again by the Chevalley formula of \cite{KNS}.

To summarize, we have the following chain of $\BZ[q^{\pm 1}][P]$-modules (and $K_{\tilde{\mathbf{I}}}(\bQGr)$ is in fact a $\BZ[P]\pra{q^{-1}}$-module):
$$ \mathbb{K}\subset \Kr\subset K_{\tilde{\mathbf{I}}}(\bQGr). $$
 Next we discuss additional structures on these modules.

\subsubsection{Heisenberg.}
Let $\hHeis$ be the $q$-Heisenberg algebra defined as a $\BZ[q^{\pm 1/e}]$-algebra by generators
\begin{align}
X^\nu\ (\nu\in P),\quad t_\mu\ (\mu\in P)
\end{align}
and relations 
\begin{align}
&X^\lambda X^\nu = X^{\lambda+\nu}\\
&t_\mu t_\xi = t_{\mu+\xi}\\
&X^0 = t_0 = 1\\
&X^\nu t_\mu = q^{(\mu,\nu)}t_\mu X^\nu.
\end{align}
for all $\lambda,\mu,\nu,\xi\in P$.

Let $\heis\subset\hHeis$ be the $\BZ[q^{\pm 1}]$-subalgebra generated by
$X^\nu\ (\nu\in P)$ and $t_\beta\ (\beta\in Q)$.

\begin{prop}\label{P:heis-act}
There exists a unique right $\heis$-module structure on $\mathbb{K}$ and $\Kr$ such that:
\begin{align}
[\CO_{\bQG(x)}(\lambda)]\cdot X^\nu &= [\CO_{\bQG(x)}(\lambda+\nu)]\\
[\CO_{\bQG(x)}(\lambda)]\cdot t_\beta &= q^{(\beta,\lambda)}\cdot[\CO_{\bQG(xt_{-w_\circ\beta})}(\lambda)]
\end{align}
for all $x\in W_\af$, $\lambda,\nu\in P$, and $\beta\in Q$. Moreover, $\mathbb{K}$ is free as an $\heis$-module, with $\heis$-basis given by $\{[\CO_{\bQG(w)}]\}_{w\in W}$.
\end{prop}

\begin{proof}
See \cite{Kat1} and \cite[Proposition 2.3]{Orr}.
We note that our conventions here differ slightly from \cite{Orr}, where the class $[\CO_{\bQG(x)}]$ is denoted by $[\overline{\CO}_{x}]$. This accounts for the twist by $-w_\circ$ in the action of $t_\beta$.
\end{proof}

\subsubsection{nil-DAHA.}

Next we turn to the nil-DAHA $\HH_0$, which is the $\BZ[\mathsf{q}^{\pm 1}]$-algebra (note the notational distinction between $\mathsf{q}$ and $q$) defined by generators
\begin{align}
T_i \ (i\in I_\af),\quad X^{\nu} \ (\nu\in P)
\end{align}
and relations
\begin{align}
&T_i T_j\dotsm = T_j T_i\dotsm \qquad\text{($m_{ij}=|s_is_j|$ factors on both sides)}\\
&T_i(T_i+1) = 0\\
&X^0=1\\
&X^{\nu}X^{\mu}=X^{\nu+\mu}\\
&T_i X^{\nu} = X^{s_i(\nu)}T_i-\frac{X^{\nu}-X^{s_i(\nu)}}{1-X^{\alpha_i}} 
\end{align}
where it is understood that $X^\delta=\mathsf{q}$.  We also use $D_i=1+T_i \ (i\in I_\af)$. These elements satisfy the braid relations as above and quadratic relations $D_i^2=D_i$.

\begin{prop}[{\cite[Proposition 6.4]{KNS}}]\label{P:daha-act}
There is a left $\HH_0$-action on $\Kr$ uniquely determined by
\begin{align}
\mathsf{q}\cdot [\CO_{\bQG(t_\alpha)}(\lambda)] &= q^{-1}[\CO_{\bQG(t_\alpha)}(\lambda)]\\
X^{\nu}\cdot[\CO_{\bQG(t_\alpha)}(\lambda)] &= e^{-\nu}[\CO_{\bQG(t_\alpha)}(\lambda)]\\
D_i\cdot [\CO_{\bQG(t_\alpha)}(\lambda)] &= [\CO_{\bQG(t_\alpha)}(\lambda)]\\
D_0\cdot [\CO_{\bQG(t_\alpha)}(\lambda)] &= [\CO_{\bQG(s_0t_\alpha)}(\lambda)]
\end{align}
for all $i\in I$, $\alpha\in Q$, and $\nu,\lambda\in P$. 
\end{prop}

We warn the reader that $X^\nu\ (\nu\in P)$ can be viewed both as an element of $\HH_0$ and as an element of $\heis$. Its left and right actions on $\Kr$ differ drastically. From the left, $X^\nu\in\HH_0$ acts as equivariant scalar multiplication by $e^{-\nu}$, while from the right, $X^\nu\in\heis$ acts as the line bundle twist by $\CO(\nu)$.

\subsubsection{$\mathbb{K}$ as a bimodule.}
By \cite[Theorem 5.1]{Orr}, the $\HH_0$-action on $\Kr$ leaves $\mathbb{K}$ stable. Thus $\mathbb{K}$ is an $(\HH_0,\heis)$-bimodule. Since $\mathbb{K}$ is free as a right $\heis$-module, the nil-DAHA action on $\mathbb{K}$ (hence on $\Kr$) is therefore characterized by the unique ring homomorphism $\varrho_0 : \HH_0 \to \Mat_W(\heis)$ satisfying
\begin{align}\label{E:rho-0}
H\cdot [\CO_{\bQG(w)}] = \sum_{v\in W}[\CO_{\bQG(v)}]\cdot \varrho_0(H)_{vw}
\end{align}
for all $H\in \HH_0$ and $w\in W$.

For us, the key point is that the matrices encoding the inverse Chevalley formula \eqref{E:inv-chev} are exactly $\varrho_0(X^{-\lambda})$ for $\lambda\in P$.

\subsubsection{Further observations.}\label{SS:further-obs}

For any $i\in I_\af$ and $x\in W_\af$, we have (see \cite{KNS} or \cite{Orr}):
\begin{align}\label{E:D-schub}
D_i\cdot [\CO_{\bQG(x)}] &=
\begin{cases}
[\CO_{\bQG(s_ix)}] &\text{if $s_ix\prec x$}\\
[\CO_{\bQG(x)}] &\text{if $s_ix\succ x$}\\
\end{cases}
\end{align}
where $\prec$ is the semi-infinite Bruhat order on $W_\af$.
A special case of \eqref{E:D-schub} is
\begin{align}\label{E:D-fin-schub}
D_i\cdot[\CO_{\bQG(w)}]=
\begin{cases}
[\CO_{\bQG(s_iw)}] &\text{if $s_iw<w$}\\
[\CO_{\bQG(w)}] &\text{if $s_iw>w$}
\end{cases}
\end{align}
for any $i\in I$ and $w\in W$, where $<$ is the Bruhat order on $W$. 

It follows from \eqref{E:D-fin-schub} and Proposition~\ref{P:daha-act} (see also \cite[Proposition 6.2]{KNS}) that the nil affine Hecke $\BZ[\mathsf{q}^{\pm 1}]$-subalgebra $\mathcal{H}_0\subset\HH_0$ generated by 
\begin{align*}
D_i \ (i\in I),\quad X^{\nu} \ (\nu\in P)
\end{align*}
leaves $K_{H\times \BC^*}(\bQG)\subset\Kr$ stable. (Our inverse Chevalley formulas make this explicit.)

We also deduce from \eqref{E:D-fin-schub} that $[\CO_{\bQG(w_\circ)}]$ generates $\mathbb{K}$ as an $(\mathcal{H}_0,\heis)$-bimodule (in fact, the action of $\heis$ and the $D_i$ for $i\in I$ on $[\CO_{\bQG(w_\circ)}]$ is sufficient to generate all of $\mathbb{K}$).

\subsection{DAHA.} \label{SS:DAHA}

The main result of \cite{Orr} gives an algebraic construction of the homomorphism $\varrho_0 : \HH_0\to\Mat_W(\heis)$ of \eqref{E:rho-0}, as a $\mathsf{t}\to 0$ limit from the polynomial representation of DAHA. Let us recall the necessary details of this construction.

\subsubsection{Definition of DAHA.}

Recall that $\kqt=\BQ(\mathsf{q}^{1/e},\mathsf{v})$ where $\mathsf{t}=\mathsf{v}^2$. The (extended) DAHA $\hHH$ is the $\kqt$-algebra defined by generators
\begin{align}
T_i\ (i\in I_\af),\quad X^\nu\ (\nu\in P),\quad \pi\ (\pi\in \Pi)
\end{align}
and relations
\begin{align}
&T_i T_j\dotsm = T_j T_i \dotsm  \qquad\text{with $m_{ij}=|s_is_j|$ factors on both sides}\\
&(T_i-\mathsf{t})(T_i+1) = 0\\
&X^\nu X^\mu=X^{\nu+\mu}\\
&X^0=1\\
&T_i X^\nu = X^{s_i(\nu)} T_i + (\mathsf{t}-1)(1-X^{\alpha_i})^{-1}(X^\nu-X^{s_i(\nu)})\label{E:X-bernstein}\\
&\pi T_i \pi^{-1} = T_j \qquad\text{where $\pi(\alpha_i)=\alpha_j$}\\
&\pi X^\nu \pi^{-1} = X^{\pi(\nu)}
\end{align}
for all $i,j\in I_\af$, $\nu,\mu\in P$, and $\pi\in \Pi$. Here it is understood that $X^{\delta/e}=\mathsf{q}^{1/e}$.

One has well-defined elements $Y^\nu\in\hHH\ (\nu\in P)$ given by
\begin{align}
Y^\nu = \mathsf{v}^{-\sum_{k=1}^\ell \epsilon_k}\pi T_{i_1}^{\epsilon_1}\dotsc T_{i_\ell}^{\epsilon_\ell}
\end{align}
for any reduced expression $t_\nu = \pi s_{i_1}\dotsc s_{i_\ell}\in P\rtimes W=W_{\mathrm{ex}}$, where  $i_k\in I_\af$ and $\pi\in\Pi$, and where $\epsilon_k\in\{\pm 1\}$ are determined by
\begin{align}
\epsilon_k = \begin{cases} +1 &\text{if $\pi s_{i_1}\dotsm s_{i_{k-1}}(\alpha_{i_k})\in-\Delta^++\BZ\delta$}\\-1 &\text{if $\pi s_{i_1}\dotsm s_{i_{k-1}}(\alpha_{i_k})\in\Delta^++\BZ\delta$}.\end{cases}
\end{align}
These elements satisfy
\begin{align}
&Y^\nu Y^\mu=Y^{\nu+\mu}\\
&Y^0=1\\
&T_i Y^\nu - Y^{s_i(\nu)} T_i = (\mathsf{t}-1)(1-Y^{-\alpha_i})^{-1}(Y^\nu-Y^{s_i(\nu)})\label{E:T-Y}
\end{align}
for any $i\in I$ and $\nu,\mu\in P$. (To extend the latter relation to $i\in I_\af$, see, e.g., \cite[\S3.2]{Orr}.)

\subsubsection{Polynomial representation.}

Let $\kqt(P)$ be the field of fractions of $\kqt[P]=\kqt[X^\nu : \nu\in P]$. The extended affine Weyl group $W_{\mathrm{ex}}$ acts on $\kqt[P]$ by automorphisms as follows: $t_\mu w(X^\nu)=q^{-(\mu,w\nu)}X^{w\nu}$. By a difference-reflection operator, we mean an element of the smash product
\begin{align*}
\kqt(P)\rtimes W_{\mathrm{ex}}\cong\kqt(P)\otimes_\kqt \kqt[W_{\mathrm{ex}}] \quad \text{($\kqt$-linear isomorphism)}
\end{align*}
acting from the left on $\kqt(P)$ by multiplication operators (first tensor factor) and the induced action of $W_{\mathrm{ex}}$ (second tensor factor). We note that $\kqt[P]\rtimes P$, where $P\cong\{t_\mu: \mu\in P\}\subset W_{\mathrm{ex}}$, is nothing but the  $q$-Heisenberg algebra $\hHeis$, with scalars extended to $\kqt$.

The polynomial representation $\varrho : \hHH\to\mathrm{End}(\kqt[P])$ is a faithful representation of $\hHH$ given as follows: $\varrho(X^\nu)$ is multiplication by $X^\nu$, the elements $\pi\in \Pi\subset W_{\mathrm{ex}}$ act by the automorphisms above, and $T_i\ (i\in I_\af)$ act by Demazure-Lusztig operators:
\begin{align}\label{E:DL}
\varrho(T_i) &= \frac{\mathsf{t}X^{\alpha_i}-1}{X^{\alpha_i}-1}s_i+\frac{1-\mathsf{t}}{X^{\alpha_i}-1}.
\end{align}

In general, elements of $\hHH$ act on $\kqt[P]$ by difference-reflection operators from $\kqt(P)\rtimes W_{\mathrm{ex}}$; as such, they are represented uniquely as sums of $ft_\mu u = f\otimes t_\mu u$, where $f\in\kqt(P)$ and $t_\mu u\in W_{\mathrm{ex}}$. Obviously, not all elements of $\kqt(P)\rtimes W_{\mathrm{ex}}$ leave $\kqt[P]$ stable, but those from $\hHH$ do. As in \eqref{E:DL}, we will simultaneously regard $\varrho(H)$ for $H\in\hHH$ as an element of $\mathrm{End}(\kqt[P])$ and as a difference-reflection operator, i.e., an element of $\kqt(P)\rtimes W_{\mathrm{ex}}$.

\subsubsection{Limiting procedure.}
Consider the homomorphism
\begin{align*}
\varkappa : \kqt(P)\rtimes W_{\mathrm{ex}} \to \Mat_W(\kqt(P)\rtimes P)
\end{align*}
which for $ft_\mu u\in\kqt(P)\rtimes W_{\mathrm{ex}}$ and $v,w\in W$ is given by
\begin{align}\label{E:kappa}
\varkappa(f t_\mu u)_{vw}=
\begin{cases}
(v^{-1}\cdot f)\big|_{X^\nu\mapsto \mathsf{t}^{-(\rho,\nu)}X^\nu}\,\mathsf{t}^{(\rho,v^{-1}\mu)}t_{v^{-1}(\mu)} &\text{if $v=uw$}\\
0 &\text{otherwise.}
\end{cases}
\end{align}
By \cite[Theorem 5.1]{Orr}, the sought-after matrices $\varrho_0(X^{-\lambda})$ from \eqref{E:rho-0} can be obtained by applying a simple automorphism to the result of the following entrywise limit:
\begin{align}\label{E:Y-lim}
\varrho'_0(Y^\lambda):=\lim_{\mathsf{v}\to 0} \varkappa(\varrho(Y^{\lambda}))\in\Mat_W(\heis).
\end{align}
The existence of this limit for any $\lambda\in P$ is ensured by \cite[Theorem 4.4]{CO}. The significance of $Y^{\lambda}$ here is that it is the image of $X^{-\lambda}$ under the DAHA duality anti-involution; see \cite[\S3.2]{Orr} for further details.

\subsection{Computing limits.}

Our goal is thus to compute certain $\varrho'_0(Y^\lambda)$ given by \eqref{E:Y-lim}. In order to do so, we will work with compositions of the following auxiliary difference-reflection operators
\begin{align}
\sG_\eta^\pm &= \frac{\mathsf{t}^{\pm 1}-X^{\eta}}{1-X^{\eta}}+\frac{\mathsf{t}^{\pm 1}-1}{1-X^{-\eta}}\, s_\eta
\end{align}
for $\eta\in\Delta^+_\af$. These are cousins of the Demazure-Lusztig operators $\varrho(T_i)$:
\begin{align*}
\sG^+_{\alpha_i} &= s_i\varrho(T_i), \qquad
\sG^-_{\alpha_i} = (\sG^+_{\alpha_i})^{-1}= \varrho(T_i^{-1})s_i.
\end{align*}

As we wish to take $\mathsf{v}\to 0$, all rational functions in such operators will be expanded in $\mathbb{Z}[\mathsf{q}^{\pm 1/e}][P]\pra{\mathsf{v}}$. In this context, we shall write $A\approx B$ to mean that $A$ and $B$ have the same lowest term, i.e., that $A$ and $B$ have the same order with respect to $\mathsf{v}$ and that $A-B$ has strictly greater order than that of $A$ and $B$.

Let $\eta\in\Delta^+$. By \eqref{E:kappa}, we see that a matrix entry of $\varkappa(\sG_\eta^\pm)$ vanishes unless it is indexed by $(w,w)$ or $(w,s_\eta w)$ for some $w\in W$. The nonvanishing entries are given as follows:
\begin{align}
\varkappa(\sG_\eta^+)_{w,w} &= 
\frac{\mathsf{t}-\mathsf{t}^{-(\rho,w^{-1}\eta)}X^{w^{-1}(\eta)}}{1-\mathsf{t}^{-(\rho,w^{-1}\eta)}X^{w^{-1}(\eta)}}
\approx
\begin{cases}
1 & \text{if $w^{-1}\eta>0$}\\
\mathsf{t}(1-X^{w^{-1}\eta}) & \text{if $(\rho,w^{-1}\eta)=-1$}\\
\mathsf{t} & \text{if $(\rho,w^{-1}\eta)<-1$}
\end{cases}\\
%
\varkappa(\sG_\eta^+)_{w,s_\eta w} &= 
\frac{\mathsf{t}-1}{1-\mathsf{t}^{(\rho,w^{-1}\eta)}X^{-w^{-1}\eta}}
\approx
\begin{cases}
-1 &\text{if $w^{-1}\eta>0$}\\
\mathsf{t}^{-(\rho,w^{-1}\eta)}X^{w^{-1}\eta} & \text{if $w^{-1}\eta<0$}
\end{cases}\\
%
\varkappa(\sG_\eta^-)_{w,w} &= 
\frac{\mathsf{t}^{-1}-\mathsf{t}^{-(\rho,w^{-1}\eta)}X^{w^{-1}(\eta)}}{1-\mathsf{t}^{-(\rho,w^{-1}\eta)}X^{w^{-1}(\eta)}}
\approx
\begin{cases}
1-X^{-w^{-1}\eta} & \text{if $(\rho,w^{-1}\eta)=1$}\\
1 & \text{if $(\rho,w^{-1}\eta)>1$}\\
\mathsf{t}^{-1} & \text{if $w^{-1}\eta<0$}
\end{cases}\\
%
\varkappa(\sG_\eta^-)_{w,s_\eta w} &= 
\frac{\mathsf{t}^{-1}-1}{1-\mathsf{t}^{(\rho,w^{-1}\eta)}X^{-w^{-1}\eta}}\approx
\begin{cases}
\mathsf{t}^{-1} &\text{if $w^{-1}\eta>0$}\\
-\mathsf{t}^{-(\rho,w^{-1}\eta)-1}X^{w^{-1}\eta} & \text{if $w^{-1}\eta<0$}.
\end{cases}
\end{align}

\begin{rem}
These computations are closely related to those of \cite[\S4]{BFP} involving quantum Bruhat operators. The precise connection to the quantum Bruhat graph will be made below.
\end{rem}

\subsubsection{Expanding products as sums over walks.} 

Using the vanishing of the matrix coefficients above, we obtain the following row expansion (see \S\ref{S:mat}) for any $w\in W$ and any sequence $\vec{\eta}=(\eta_1,\dotsc,\eta_n)\in(\Delta^+)^n$:
\begin{align}\label{E:G-matrix}
\varkappa(\sG_{\eta_1}^\pm\sG_{\eta_2}^\pm\dotsm \sG_{\eta_n}^\pm)_{w,\bullet}
&=\sum_{\mathbf{w}\in\WalkStd}\varkappa(\sG_{\eta_1}^\pm)_{w_0,w_1}\varkappa(\sG_{\eta_2}^\pm)_{w_1,w_2}\dotsm\varkappa(\sG_{\eta_n}^\pm)_{w_{n-1},w_n}e_{w_n},
\end{align}
where $w_0=w$ and $\WalkStd$ is the set of sequences $\mathbf{w}=(w_1,\dotsc,w_n)$ of elements of $W$ such that \begin{align}
\text{$w_t\in\{w_{t-1},s_{\eta_t}w_{t-1}\}$,\quad for each $t=1,\dotsc,n$.}
\end{align}
We call elements of $\WalkStd$ {\em walks} in $W$.

\subsubsection{Quantum walks.}\label{SS:QW}

Below we will see that the matrices describing our inverse Chevalley formula are given by sums over subsets of walks given as follows:

\begin{dfn}
We call $\mathbf{w}=(w_1,\dotsc,w_n)\in\WalkStd$ a {\em quantum walk} if $\mathbf{w}$ defines a directed walk in the quantum Bruhat graph $\QBG(W)$, i.e., for each $1\le t\le n$, either $w_t=w_{t-1}$ or $w_{t-1}\to w_t=s_{\eta_t}w_{t-1}$ is an edge in $\QBG(W)$.
\end{dfn}

We note that such an edge $w_{t-1}\to w_t$ in $\QBG(W)$ is labeled by $\pm w_{t-1}^{-1}\eta_t$. In the case when $w_{t-1}>w_t$, the requirement for an edge is that $\ell(w_{t-1})-\ell(w_t)=(2\rho,-w_{t-1}^{-1}\eta_t)-1$.

\subsubsection{Main limit.}\label{S:main-limit}

Now we proceed to the limit computation which will give our first main result. Let $\varpi_k\in P_+$ be a minuscule fundamental weight. (Thus we implicitly assume that our simply-laced group $G$ is not of type $E_8$.) As in the introduction, let $J=I\setminus\{k\}$, let $W_J=\langle s_j \mid j\in J\rangle \subset W$ be the corresponding parabolic subgroup, which is the stabilizer of $\varpi_k$, and let $W^J$ denote the set of minimal coset representatives for $W/W_J$. For $w\in W$, let $\mcr{w}\in W^J$ be the minimal representative of its coset $wW_J$.

Fix $x\in W^J$ and let $\lambda=x\varpi_k\in P$, which is an arbitrary minuscule weight.
Let $y \in W$ be the unique element such that $\mcr{\lng} = yx$. Since $\vpi_{k}$ is minuscule, we have
$\ell(\mcr{\lng}) = \ell(y) + \ell(x)$; indeed, on $W^J$, the Bruhat order coincides with 
the left weak Bruhat order \cite[Lemma 11.1.16]{Gr13}.

We fix reduced expressions $x=s_{j_l}\dotsm s_{j_1}$ and $y = s_{i_{1}} \cdots s_{i_{m}}$, and we define:
\begin{align}
\beta_r &= s_{j_{l}}s_{j_{l-1}}
\cdots s_{j_{r+1}}(\alpha_{j_{r}}), \qquad 
\text{for $1 \le r \le l$}, \\
\gamma_s &=s_{i_{m}}s_{i_{m-1}}
\cdots s_{i_{s+1}}(\alpha_{i_{s}}), \qquad 
\text{for $1 \le s \le m$}.
\end{align}
Hence, if we set
\begin{align}
\Inv(w)=\Delta^+\cap w(-\Delta^+),
\end{align}
then $\Inv(x)=\{\beta_1,\dots,\beta_l\}$ and $\Inv(y^{-1})=\bigl\{\gamma_{1},\dots,\gamma_{m}\bigr\}$.

Now, for any $w\in W$, define the set of walks $\Walk_{\lambda,w}=\WalkStd$ where $\vec{\eta}$ is given by
\begin{align}
\vec{\eta}=(\eta_1,\dotsc,\eta_n)=(\beta_l,\dotsc,\beta_1,\gamma_1,\dotsc,\gamma_m)
\end{align}
and $n=l+m$. Let $\QWalk_{\lambda,w}$ denote the subset of quantum walks in $\WalkStd$. (These sets depend on the choice of reduced expressions for $x$ and $y$, even though our notation does not indicate this.)

The inverse Chevalley formula for minuscule weights is an immediate consequence of the following, which is the main technical achievement of this paper:

\begin{thm}\label{T:Y-lim}
For any minuscule $\lambda\in P$ and $w\in W$, we have
\begin{align}
\varrho_0'(Y^{\lambda})_{w,\bullet}=\sum_{\mathbf{w}=(w_1,\dotsc,w_n)\in\QWalk_{\lambda,w}}g^-_{\mathbf{w}}\, t_{w_{l}^{-1}\lambda}\, g^+_{\mathbf{w}}\, e_{w_n}
\end{align}
where
\begin{align}
g^-_{\mathbf{w}}
&=
\prod_{1\le t\le l}
\begin{cases}
1-X^{-w_{t-1}^{-1}\eta_t} & \text{if $w_t=w_{t-1}$ and $(\rho,w_{t-1}^{-1}\eta_t)=1$}\\
-X^{w_{t-1}^{-1}\eta_t} & \text{if $w_t<w_{t-1}$}\\
1 & \text{otherwise},
\end{cases}\\
g^+_{\mathbf{w}}
&=
\prod_{l<t\le n}
\begin{cases}
1-X^{w_{t-1}^{-1}\eta_t} & \text{if $w_t=w_{t-1}$ and $(\rho,w_{t-1}^{-1}\eta_t)=-1$}\\
X^{w_{t-1}^{-1}\eta_t} & \text{if $w_t<w_{t-1}$}\\
-1 & \text{if $w_t>w_{t-1}$}\\
1 & \text{otherwise}.
\end{cases}
\end{align}
\end{thm}

For the proof of Theorem~\ref{T:Y-lim}, we will need the following notions, which rely heavily on $\lambda\in P$ being minuscule. For $w \in W$, and with all other notation as above, let us define
\begin{align}
& 
\Inv(w)_{\lambda}^{+}=\bigl\{ \gamma \in \Inv(w) \mid (\gamma,\lambda)=1\bigr\},
\qquad \ell_{\lambda}^{+}(w)=\# \Inv(w)_{\lambda}^{+}, \\
& 
\Inv(w)_{\lambda}^{-}=\bigl\{ \gamma \in \Inv(w) \mid (\gamma,\lambda)=-1\bigr\},
\qquad \ell_{\lambda}^{-}(w)=\# \Inv(w)_{\lambda}^{-},
\end{align}
and
\begin{align} 
\Inv(w)_{\lambda,r}^{-}&=\Inv(w) \cap 
\bigl\{\beta_{r},\beta_{r+1},\dots,\beta_{l}\bigr\}, \qquad 
\ell_{\lambda,r}^{-}(w)=\#\Inv(w)_{\lambda,r}^{-},\\
\Inv(w)_{\lambda,s}^{+}&=\Inv(w) \cap 
\bigl\{\gamma_{s},\gamma_{s+1},\dots,\gamma_{m}\bigr\}, \qquad 
\ell_{\lambda,s}^{+}(w)=\#\Inv(w)_{\lambda,s}^{+}, 
\end{align}
where $1 \le r \le l$ and  $1 \le s \le m$.

\begin{prop}\label{P:L}
With all notation as above:
\begin{enumerate}
\item We have
\begin{align}\label{E:l-x-lng}
2\ell(x) - \ell(\mcr{\lng}) &= - 2(\rho,\lambda).
\end{align}
\item For any $w\in W$,
\begin{align}\label{E:rho-w-rho}
(\rho-w\rho,\lambda) &= \ell_{\lambda}^{+}(w) - \ell_{\lambda}^{-}(w),
\end{align}
and
\begin{align}
\Inv(w)_{\lambda}^{-} &= \Inv(w) \cap \Inv(x).\label{E:inv-la-}\\
\Inv(w)_{\lambda}^{+} &= \Inv(w) \cap \Inv(y^{-1})\label{E:inv-la+}.
\end{align}

\item If $s_{\beta_{r}}w < w$ for $1 \le r \le l$, then 
\begin{equation}\label{E:l-diff-}
\ell(w)-\ell(s_{\beta_{r}}w) = 
2 (\ell_{\lambda,r}^{-}(w)-\ell_{\lambda,r}^{-}(s_{\beta_{r}}w)) - 1.
\end{equation}
\item If $s_{\gamma_{s}}w < w$ for $1 \le s \le m$, then
\begin{equation}\label{E:l-diff+}
\ell(w)-\ell(s_{\gamma_{s}}w) = 
2 (\ell_{\lambda,s}^{+}(w)-\ell_{\lambda,s}^{+}(s_{\gamma_{s}}w))- 1.
\end{equation}
\end{enumerate}
\end{prop}

\begin{proof}
See Appendix~\ref{A:L}.
\end{proof}

\begin{rem}
Formulas \eqref{E:l-diff-} and \eqref{E:l-diff+} generalize the following well-known formula for length differences in the symmetric group $S_n$:
\begin{align*}
\ell(w)-\ell((ij)w) &= 1+2\cdot|\{k: \text{$i<k<j$ and $w^{-1}(i)>w^{-1}(k)>w^{-1}(j)$}\}|
\end{align*}
where $1\le i < j\le n$, $w\in S_n$ is a permutation such that $w^{-1}(i)>w^{-1}(j)$, and $(ij)$ denotes the transposition swapping $i$ and $j$.
\end{rem}

\begin{proof}[Proof of Theorem~\ref{T:Y-lim}]
Let $\pi\in \Pi$ be the length-zero element of $W_{\mathrm{ex}}$ corresponding to coset $\varpi_k+Q$ in $P/Q$. We have the formula $t_{\varpi_k} = \pi \mcr{\lng}=\pi yx$ (which may be taken as the definition of $\pi$) and hence $t_{\lambda} = x\pi y.$

Using \eqref{E:T-Y} and \eqref{E:l-x-lng}, we find
\begin{align*}
Y^{\lambda}
&= \mathsf{v}^{2\ell(x)-\ell(\mcr{\lng})} T_{x^{-1}}^{-1} \pi T_y\\
&= \mathsf{t}^{-(\rho,\lambda)}T_{x^{-1}}^{-1} \pi T_y
\end{align*}
in $\hHH$, and
\begin{align*}
\varrho(Y^{\lambda})
&= \mathsf{t}^{-(\rho,\lambda)} \varrho(T_{x^{-1}}^{-1})x^{-1} t_{\lambda}y^{-1} \varrho(T_y)\\
&= \mathsf{t}^{-(\rho,\lambda)}\sG^-_{\beta_l}\dotsm \sG^-_{\beta_1}t_{\lambda}\sG^+_{\gamma_1}\dotsm \sG^+_{\gamma_m}
\end{align*}
in the polynomial representation.
Thus we need to compute the limit
\begin{align*}
\varrho_0'(Y^{\lambda})=\lim_{\mathsf{v}\to 0} \varkappa(\sG^-_{\beta_l}\dotsm \sG^-_{\beta_1})\,\mathsf{t}^{-(\rho,\lambda)}\varkappa(t_{\lambda})\,\varkappa(\sG^+_{\gamma_1}\dotsm \sG^+_{\gamma_m}).
\end{align*}

The matrix $\mathsf{t}^{-(\rho,\lambda)}\varkappa(t_{\lambda})$ is diagonal with entries
\begin{align*}
\mathsf{t}^{-(\rho,\lambda)}\varkappa(t_{\lambda})_{w,w} &= 
\mathsf{t}^{(\rho,w^{-1}\lambda)-(\rho,\lambda)}t_{w^{-1}\lambda}. 
\end{align*}
Using \eqref{E:rho-w-rho}, we can write
$\mathsf{t}^{-(\rho,\lambda)}\varkappa(t_{\lambda})=\varkappa_\lambda^-\tau_\lambda\varkappa_\lambda^+$, where $\varkappa_\lambda^\pm$ and $\tau_\lambda$ are the diagonal matrices given by
\begin{align}
(\varkappa_\lambda^-)_{w,w}&=\mathsf{t}^{\ell_\lambda^-(w)},\qquad
(\tau_\lambda)_{w,w} = t_{w^{-1}\lambda},\qquad
(\varkappa_\lambda^+)_{w,w}=\mathsf{t}^{-\ell_\lambda^+(w)}.
\end{align}
Our strategy is to commute $\varkappa_\lambda^-$ to the left past $\varkappa(\sG^-_{\beta_l}\dotsm \sG^-_{\beta_1})$, and $\varkappa_\lambda^+$ to the right past $\varkappa(\sG^+_{\gamma_1}\dotsm \sG^+_{\gamma_m})$. After doing so, all negative powers of $\mathsf{v}$ will disappear.

For $r=1,\dotsc,l$, we have
\begin{align*}
\varkappa(\sG_{\beta_r}^-)_{w,w} \cdot \mathsf{t}^{\ell_{\lambda,r}^-(w)} &\approx
\mathsf{t}^{\ell_{\lambda,r+1}^-(w)}\times
\begin{cases}
1-X^{-w^{-1}\beta_r} & \text{if $(\rho,w^{-1}\beta_r)=1$}\\
1 & \text{if $(\rho,w^{-1}\beta_r)>1$}\\
1 & \text{if $w^{-1}\beta_r<0$}
\end{cases}
\end{align*}
and
\begin{align*}
&\varkappa(\sG_{\beta_r}^-)_{w,s_{\beta_r}w} \cdot \mathsf{t}^{\ell_{\lambda,r}^-(s_{\beta_r}w)}\\
&\approx
\mathsf{t}^{\ell_{\lambda,r+1}^-(w)}
\times
\begin{cases}
\mathsf{t}^{-1+\ell_{\lambda,r}^-(s_{\beta_r}w)-\ell_{\lambda,r}^-(w)} &\text{if $w^{-1}\beta_r>0$}\\
-\mathsf{t}^{-(\rho,w^{-1}\beta_r)+\ell_{\lambda,r}^-(s_{\beta_r}w)-\ell_{\lambda,r}^-(w)}X^{w^{-1}\beta_r} & \text{if $w^{-1}\beta_r<0$}
\end{cases}\\
&=
\mathsf{t}^{\ell_{\lambda,r+1}^-(w)}
\times
\begin{cases}
\mathsf{v}^{-1+\ell(s_{\beta_r}w)-\ell(w)} &\text{if $w^{-1}\beta_r>0$}\\
-\mathsf{v}^{-(2\rho,w^{-1}\beta_r)-1+\ell(s_{\beta_r}w)-\ell(w)}X^{w^{-1}\beta_r} & \text{if $w^{-1}\beta_r<0$}
\end{cases}
\end{align*}
where we use \eqref{E:l-diff-} to obtain the last equality. Notice that, in all cases, each factor after the bracket involves no negative powers of $\mathsf{v}$. Moreover, after commuting $\varkappa_\lambda^-$ all the way to the left, we arrive at $\mathsf{t}^{\ell_{\lambda,l+1}^-(w)}=1$, for any $w\in W$.

Similarly, for $s=1,\dotsc,m$, we have
\begin{align*}
\mathsf{t}^{-\ell_{\lambda,s}^+(w)}\cdot \varkappa(\sG_{\gamma_s}^+)_{w,w} & 
\approx
\mathsf{t}^{-\ell_{\lambda,s+1}^+(w)}\times
\begin{cases}
1 & \text{if $w^{-1}\gamma_s>0$}\\
1-X^{w^{-1}\gamma_s} & \text{if $(\rho,w^{-1}\gamma_s)=-1$}\\
1 & \text{if $(\rho,w^{-1}\gamma_s)<-1$}
\end{cases}
\end{align*}
and
\begin{align*}
&\mathsf{t}^{-\ell_{\lambda,s}^+(w)}\cdot \varkappa(\sG_{\gamma_s}^+)_{w,s_{\gamma_s} w}\\
&\approx
\mathsf{t}^{-\ell_{\lambda,s+1}^+(s_{\gamma_s}w)}
\times
\begin{cases}
-\mathsf{t}^{-1+\ell_{\lambda,s}^+(s_{\gamma_s}w)-\ell_{\lambda,s}^+(w)}&\text{if $w^{-1}\gamma_s>0$}\\
\mathsf{t}^{-(\rho,w^{-1}\gamma_s)+\ell_{\lambda,s}^+(s_{\gamma_s}w)-\ell_{\lambda,s}^+(w)}X^{w^{-1}\gamma_s} & \text{if $w^{-1}\gamma_s<0$}
\end{cases}\\
&=
\mathsf{t}^{-\ell_{\lambda,s+1}^+(s_{\gamma_s}w)}
\times
\begin{cases}
-\mathsf{v}^{-1+\ell(s_{\gamma_s}w)-\ell(w)} & \text{if $w^{-1}\gamma_s>0$}\\
\mathsf{v}^{-(2\rho,w^{-1}\gamma_s)-1+\ell(s_{\gamma_s}w)-\ell(w)}X^{w^{-1}\gamma_s} & \text{if $w^{-1}\gamma_s<0$}.
\end{cases}
\end{align*}
As above, each factor after the bracket involves no negative powers of $\mathsf{v}$, and after commuting $\varkappa_\lambda^+$ all the way to the right, we arrive at $\mathsf{t}^{\ell_{\lambda,m+1}^+(w)}=1$, for any $w\in W$.

To complete the proof, we expand the product of resulting matrices (including $\tau_\lambda$) and take $\mathsf{v}\to 0$. Taking into account the exponents of $\mathsf{v}$ above, one sees that the surviving terms are exactly those indexed by quantum walks $\QWalk_{\lambda,w}$.
\end{proof}

\subsection{Inverse Chevalley formula.}

By means of \cite[Theorem 5.1]{Orr}, Theorem~\ref{T:Y-lim} immediately gives the following algebraic form of the inverse Chevalley formula in $K_{H\times\BC^*}(\bQG)$, which is our first main result:
\begin{thm}\label{T:K-w-arb}
For any minuscule $\lambda\in P$ and $w\in W$, we have
\begin{align}\label{E:K-w-arb}
e^\lambda\cdot[\CO_{\bQG(w)}] &= \sum_{\mathbf{w}=(w_1,\dotsc,w_n)\in\QWalk_{\lambda,w}}[\CO_{\bQG(w_n)}]\cdot  \tg^+_{\mathbf{w}} \, X^{-\lng w_l^{-1}\lambda}\, \tg^-_{\mathbf{w}},
\end{align}
in terms of the $q$-Heisenberg action of Proposition~\ref{P:heis-act}, where
\begin{align}\label{E:tg-}
\tg^-_{\mathbf{w}}
&=
\prod_{1\le t\le l}
\begin{cases}
1-\tX^{-\lng w_{t-1}^{-1}\eta_t} & \text{if $w_t=w_{t-1}$ and $(\rho,w_{t-1}^{-1}\eta_t)=1$}\\
-\tX^{\lng w_{t-1}^{-1}\eta_t} & \text{if $w_t<w_{t-1}$}\\
1 & \text{otherwise}
\end{cases}\\
\label{E:tg+}
\tg^+_{\mathbf{w}}
&=
\prod_{l< t\le n}
\begin{cases}
1-\tX^{\lng w_{t-1}^{-1}\eta_t} & \text{if $w_t=w_{t-1}$ and $(\rho,w_{t-1}^{-1}\eta_t)=-1$}\\
\tX^{\lng w_{t-1}^{-1}\eta_t} & \text{if $w_t<w_{t-1}$}\\
-1 & \text{if $w_t>w_{t-1}$}\\
1 & \text{otherwise}
\end{cases}
\end{align}
and $\tX^\beta=q\cdot t_\beta X^\beta$ for $\beta\in\Delta$.
\end{thm}

\begin{rem}
The elements $\tX^\beta$ commute, as one easily checks.
\end{rem}

\begin{rem}
The truncation of \eqref{E:K-w-arb} to classes supported on the Schubert varieties $\{\bQG(w)\}_{w\in W}$ is achieved by setting $\tX^\beta=0$ in \eqref{E:tg-} and \eqref{E:tg+}, where $\beta\in\Delta^+$ always. One recovers the corresponding inverse Chevalley form in $K_H(G/B)$. Note that the truncation of \eqref{E:K-w-arb} is positive for $\lambda\in -P_+$ (when $l=n$) and alternating for $\lambda\in P_+$ (when $l=0$).
\end{rem}

\subsubsection{Decorated quantum walks.}\label{SS:DQW}
We now proceed to further combinatorialize \eqref{E:K-w-arb}. First, we enlarge the summation set, so that each term on the right-hand side of \eqref{E:K-w-arb} is a product of monomials in the $q$-Heisenberg algebra.

Given $\mathbf{w}=(w_1,\dotsc,w_n)\in\QWalk_{\lambda,w}$, let 
$S^-(\mathbf{w})$ denote the set of steps $t$, for $1\le t\le l$, such that $w_t=w_{t-1}$ and $(\rho,w_{t-1}^{-1}\eta_t)=1$. Similarly, let $S^+(\mathbf{w})$ denote the set of steps $t$, for $l<t\le n$ such that $w_t=w_{t-1}$ and $(\rho,w_{t-1}^{-1}\eta_t)=-1$. 

Let $S(\mathbf{w})=S^-(\mathbf{w})\cup S^+(\mathbf{w})$,
and define the set of {\em decorated quantum walks} $\DQWalk_{\lambda,w}$ to consist of all pairs $(\mathbf{w},\mathbf{b})$ where $\mathbf{w}\in \QWalk_{\lambda,w}$ and $\mathbf{b}$ is a $\{0,1\}$-valued function on $S(\mathbf{w})$. Then \eqref{E:K-w-arb} can be written as:
\begin{align}\label{E:K-w-arb-heis}
e^{\lambda}\cdot[\CO_{\bQG(w)}] &= \sum_{(\mathbf{w},\mathbf{b})\in\DQWalk_{\lambda,w}}[\CO_{\bQG(w_n)}]\cdot  \tg^+_{(\mathbf{w},\mathbf{b})} \, X^{-\lng w_l^{-1}\lambda}\, \tg^-_{(\mathbf{w},\mathbf{b})},
\end{align}
where
\begin{align}
\tg^-_{(\mathbf{w},\mathbf{b})}
&=
\prod_{1\le t\le l}
\begin{cases}
(-\tX^{-\lng w_{t-1}^{-1}\eta_t})^{\mathbf{b}(t)} & \text{if $t\in S^-(\mathbf{w})$}\\
-\tX^{\lng w_{t-1}^{-1}\eta_t} & \text{if $w_t<w_{t-1}$}\\
1 & \text{otherwise},
\end{cases}\\
\tg^+_{(\mathbf{w},\mathbf{b})}
&=
\prod_{l< t\le n}
\begin{cases}
(-\tX^{\lng w_{t-1}^{-1}\eta_t})^{\mathbf{b}(t)} & \text{if $t\in S^+(\mathbf{w})$}\\
\tX^{\lng w_{t-1}^{-1}\eta_t} & \text{if $w_t<w_{t-1}$}\\
-1 & \text{if $w_t>w_{t-1}$}\\
1 & \text{otherwise}.
\end{cases}
\end{align}

Next, in order to expand \eqref{E:K-w-arb-heis} further, let us introduce some more notation. For $(\mathbf{w},\mathbf{b})\in\DQWalk_{\lambda,w}$, define the sign
\begin{align*}
(-1)^{(\mathbf{w},\mathbf{b})} &= \prod_{\substack{1\le t\le l\\w_t<w_{t-1}}}(-1)\prod_{\substack{l< t\le n \\ w_t>w_{t-1}}}(-1)\prod_{t\in S(\mathbf{w})}(-1)^{\mathbf{b}(t)}
\end{align*}
and partial weights (in $Q$):
\begin{align*}
\wt_0(\mathbf{w},\mathbf{b})&=0\\
\wt_t(\mathbf{w},\mathbf{b})&=\wt_{t-1}(\mathbf{w},\mathbf{b})+
\begin{cases}
-\mathbf{b}(t)\lng w_{t-1}^{-1}\eta_t & \text{if $t\in S^-(\mathbf{w})$}\\
\lng w_{t-1}^{-1}\eta_t & \text{if $w_t<w_{t-1}$}\\
0 & \text{otherwise}
\end{cases},
\qquad \text{for $1\le t\le l$}\\
\wt_t(\mathbf{w},\mathbf{b})&=\wt_{t-1}(\mathbf{w},\mathbf{b})+
\begin{cases}
\mathbf{b}(t)\lng w_{t-1}^{-1}\eta_t & \text{if $t\in S^+(\mathbf{w})$}\\
\lng w_{t-1}^{-1}\eta_t & \text{if $w_t<w_{t-1}$}\\
0 & \text{otherwise}.
\end{cases},
\qquad \text{for $l<t\le n$}.
\end{align*}
Define the weight of $(\mathbf{w},\mathbf{b})$ to be $\wt(\mathbf{w},\mathbf{b})=\wt_n(\mathbf{w},\mathbf{b})\in Q$, and set $d_t(\mathbf{w},\mathbf{b})=\wt_t(\mathbf{w},\mathbf{b})-\wt_{t-1}(\mathbf{w},\mathbf{b})$ for $1\le t\le n$.

Define the partial degrees (in $\BZ$):
\begin{align*}
\deg_0^-(\mathbf{w},\mathbf{b}) &= 0\\
\deg_t^-(\mathbf{w},\mathbf{b}) &= \deg_{t-1}^-(\mathbf{w},\mathbf{b})+\frac{|d_t(\mathbf{w},\mathbf{b})|^2}{2}+(d_t(\mathbf{w},\mathbf{b}),\wt_{t-1}(\mathbf{w},\mathbf{b})),\quad \text{for $1\le t\le l$}\\
\deg_{l}^+(\mathbf{w},\mathbf{b}) &= \deg_{l}^-(\mathbf{w},\mathbf{b})+(-\lng w_{l}^{-1}\lambda,\wt_{l}(\mathbf{w},\mathbf{b}))\\
\deg_t^+(\mathbf{w},\mathbf{b}) &= \deg_{t-1}^+(\mathbf{w},\mathbf{b})+\frac{|d_t(\mathbf{w},\mathbf{b})|^2}{2}+(d_t(\mathbf{w},\mathbf{b}),\wt_{t-1}(\mathbf{w},\mathbf{b})), \quad \text{for $l<t\le n$}.
\end{align*}
Define $\deg(\mathbf{w},\mathbf{b})=\deg_n^+(\mathbf{w},\mathbf{b})\in\BZ$.

Then, from \eqref{E:K-w-arb-heis}, we obtain our second main result, the combinatorial form of our inverse Chevalley formula:
\begin{thm}\label{T:K-w-arb-dec}
For any minuscule $\lambda\in P$ and $w\in W$, we have
\begin{align}\label{E:K-w-arb-dec}
&e^\lambda\cdot[\CO_{\bQG(w)}]\\
&\quad = \sum_{(\mathbf{w},\mathbf{b})\in\DQWalk_{\lambda,w}}(-1)^{(\mathbf{w},\mathbf{b})}q^{\deg(\mathbf{w},\mathbf{b})}\cdot[\CO_{\bQG(w_nt_{-\lng(\wt(\mathbf{w},\mathbf{b}))})}(-\lng w_l^{-1}\lambda+\wt(\mathbf{w},\mathbf{b}))].\notag
\end{align}
\end{thm}

\begin{proof}
Proceed from right to left in \eqref{E:K-w-arb-heis}, commuting all translations to the left of line bundle twists, and then act on $[\CO_{\bQG(w_n)}]$.
\end{proof}


\appendix

\section{The type $A$ case.}\label{A:A}

We briefly consider the type $A$ case. For this appendix, let $G=SL(n+1,\BC)$. 

\subsection{The root system of type $A$.} 

Let $\{\varepsilon_i : 1\le i\le n+1\}$ be the standard basis of $\BZ^{n+1}$.
We realize the weight lattice as $P=\BZ^{n+1}/\BZ(\varepsilon_1+\dotsm+\varepsilon_{n+1})$, and, by abuse of notation, we continue to denote the image of $\varepsilon_i$ in $P$ by the same symbol. 
Thus $\varpi_k := \varepsilon_1+\dotsm+\varepsilon_k$, for $k \in \{1, \ldots, n\}$, are the fundamental weights of $G$.

We set $\alpha_{i} := \ve_{i} - \ve_{i+1}$ for $i \in \{ 1, \ldots, n \}$ 
and $\alpha_{i, j} := \alpha_{i} + \alpha_{i+1} + \cdots + \alpha_{j}$ for $i, j \in \{1, \ldots, n\}$ with $i \le j$. 
Then the root system is $\Delta := \{ \pm \alpha_{i, j} \mid 1 \le i \le j \le n \}$, 
with the set of positive roots $\Delta^+ := \{ \alpha_{i, j} \mid 1 \le i \le j \le n \}$, 
and the set of simple roots $\{ \alpha_{1}, \ldots, \alpha_{n} \}$. 
 
We identify the Weyl group $W$ with the symmetric group $\mathfrak{S}_{n+1}$ in the usual way.
Regarding $w\in W$ as a permutation, we have $w \ve_{i} = \ve_{w(i)}$ for $i \in \{ 1, \ldots, n+1 \}$. 
Thus, for $i, j \in \{ 1, \ldots, n \}$ with $i \leq j$, the reflection $s_{\alpha_{i,j}}$ corresponds to the transposition $(i, j+1)\in \mathfrak{S}_{n+1}$. The longest element of $W$ is given by $\lng(i) = n+2-i$ for $i \in \{ 1, \ldots, n+1 \}$. 


\subsection{Generators.}

In type $A$, each fundamental weight $\varpi_k \ (k=1,\dotsc,n)$ is minuscule. Moreover, we have $\lng \varpi_k=-\varpi_{n+1-k}$. 
In fact, since the $\varepsilon_i$ generate $P$, the inverse Chevalley rule in $K_{H\times\BC^*}(\bQGr)$ is completely determined by Theorem~\ref{T:K-w-arb-dec} in the case of $\lambda=\pm \varepsilon_i$, where $1\le i\le n+1$. (That is, $\lambda$ belongs to the $W$-orbit of either $\varpi_1$ or $\varpi_n=-\varepsilon_{n+1}$.)

\subsection{Walks.}

Let us give an explicit description of the sequences of positive roots defining the sets $\Walk_{\lambda,w}$ and $\QWalk_{\lambda,w}$ in the case $\lambda=\varepsilon_{l+1}$ for $0\le l\le n$. (The case when $\lambda$ belongs to the $W$-orbit of $\varpi_n$ is similar, and can be obtained by a diagram automorphism.) In the setting of \S\ref{S:main-limit}, we have $k=1$ and
\begin{align}
\mcr{\lng}&=s_n \dotsm s_1\\
&=\underbrace{s_{n}\dotsm s_{l+1}}_{=y}\underbrace{s_l\dotsm s_1}_{=x}.
\end{align}
For any $w\in W$, we take the set of walks $\Walk_{\varepsilon_{l+1},w}=\Walk_w^{\vec{\eta}}$ for $\vec{\eta}=(\eta_1,\dotsc,\eta_n)=(\beta_l,\dotsc,\beta_1,\gamma_1,\dotsc,\gamma_m)$, where $m=\ell(y)=n-l$ and 
\begin{align}
\beta_r &=\varepsilon_r-\varepsilon_{l+1},\quad \text{for $1\le r\le l$},\\
\gamma_s &=\varepsilon_{l+1} - \varepsilon_{n+2-s},\quad \text{for $1\le s\le m$}.
\end{align}
Let us abbreviate our notation as follows: $\Walk_w^{(l)}=\Walk_{\varepsilon_{l+1},w}$ and $\QWalk_w^{(l)}=\QWalk_{\varepsilon_{l+1},w}$. 



\subsection{Special case: $w=w_\circ$.}

The following lemma describes the set $\QWalk^{(l)}_{w_\circ}$. Its proof is straightforward and is left to the reader.

\begin{lem}\label{L:QW-w0}
A walk $\mathbf{w}\in\Walk_{w_\circ}^{(l)}$ belongs to $\QWalk^{(l)}_{w_\circ}$
if and only if one of the following holds:
\begin{enumerate}
\item[(a)] $\mathbf{w}=(w_1,\dotsc,w_l,w_{l},\dotsc,w_{l})$ where $(w_1,\dotsc,w_{l})\in \Walk_{w_\circ}^{(\alpha_{l,l},\dotsc,\alpha_{1,l})}$
\item[(b)] $\mathbf{w}=(w_\circ,\dotsc,w_\circ,w_{l+1},\dotsc,w_n)$ where $(w_{l+1},\dotsc,w_n)\in\Walk_{w_\circ}^{(\alpha_{l+1,n},\dotsc,\alpha_{l+1,l+1})}$
\end{enumerate}
We have $S^-(\mathbf{w})=\varnothing$ in all cases, and $S^+(\mathbf{w})=\varnothing$ unless $l<n$ and $\mathbf{w}=(w_\circ,\dotsc,w_\circ)$, in which case $S^+(\mathbf{w})=\{n\}$. (Note that the two cases above share the walk $\mathbf{w}=(w_\circ,\dotsc,w_\circ)$.)
\end{lem}


Using Lemma~\ref{L:QW-w0}, one can show the following (where we set $i=l+1$):

\begin{prop}\label{P:K-w0}
In $K_{H\times\BC^*}(\bQG)$ for $G=SL(n+1,\BC)$, one has
\begin{align}\label{E:K-w0}
e^{\varepsilon_{i}}\cdot[\CO_{\bQG(w_\circ)}]
&=
[\CO_{\bQG(w_\circ)}(-\varepsilon_{i})]-\mathbf{1}_{\{i<n+1\}}\cdot q\cdot [\CO_{\bQG(w_\circ t_{-w_\circ(\alpha_i)})}(-\varepsilon_{i+1})]\\
&\ \ +\sum_{\varnothing\neq\{i_1<\dotsm<i_a\}\subset\{1,\dotsc,i-1\}}(-1)^a [\CO_{\bQG((i_1\dotsm i_ai)^{-1}w_\circ t_{-w_\circ(\alpha_{i_1,i-1})})}(-\varepsilon_{i})]\notag\\
&\ \ +\sum_{\varnothing\neq\{j_1<\dotsm<j_{b}\}\subset \{i+1,n+1\}} (-1)^{b-1}q\cdot[\CO_{\bQG((i j_1\dotsm j_b)^{-1}w_\circ t_{-w_\circ(\alpha_{i,j_b-1})})}(-\varepsilon_{j_b})]\notag
\end{align}
where $1\le i\le n+1$ and
\begin{align}\label{E:indicator-def}
\mathbf{1}_{\{i<n+1\}}=\begin{cases} 1 &\text{if $i<n+1$}\\ 0 &\text{otherwise.}\end{cases}
\end{align}
\end{prop}

\begin{rem}
The leading term $[\CO_{\bQG(w_\circ)}(-\varepsilon_i)]$ in \eqref{E:K-w0} recovers the ``classical'' analog of this formula in $K_H(G/B)$, while the remaining terms give the explicit corrections in $K_{H\times\BC^*}(\bQG)$.
\end{rem}

\begin{rem}\label{R:A}
Applying the diagram automorphism ``$-w_\circ$,'' which sends $[\CO_{\bQG(wt_\beta)}(\mu)]\mapsto[\CO_{\bQG(w_\circ ww_\circ t_{-w_\circ(\beta)})}(-w_\circ\mu)]$ and an equivariant scalar $e^\lambda\mapsto e^{-w_\circ\lambda}$, one obtains a similar formula for $e^{-\varepsilon_{n+2-i}}\cdot[\CO_{\bQG(\lng)}]$ in $K_{H\times\BC^*}(\bQG)$. By the observations in $\S\ref{SS:further-obs}$, the action of $e^\lambda$ on $K_{H\times\BC^*}(\bQG)$ for {\em any} $\lambda\in P$  is then completely determined by iteration of these two formulas together with the $D_i$ for $i\in I$. 
\end{rem}

\subsection{Symmetrization.}\label{SS:sym}

The usual $q$-Toda difference operators are realized geometrically by the ``spherical part'' of $\mathbb{K}$, which is the free $\heis$-submodule generated by $[\CO_{\bQG}]$. By \cite[Corollary 5.3]{Orr}, the action (of a spherical nil-affine Hecke subalgebra of $\HH_0$) on the spherical part is obtained by taking the $(e,e)$-entry of the corresponding matrix in $\Mat_W(\heis)$, namely:
\begin{align}\label{E:sph-act}
f(X)\cdot[\CO_{\bQG}]=[\CO_{\bQG}]\cdot\varrho_0(f)_{e,e}
\end{align}
for any symmetric Laurent polynomial $f=f(X)\in\mathbb{Z}[q^{\pm 1}][X]^W\subset\HH_0$.

Applying \eqref{E:sph-act} for $G=SL(n+1,\BC)$ to the symmetrization of $e^{\varpi_1}$, one obtains
\begin{align}
\sum_{i=1}^{n+1}e^{\varepsilon_i}\cdot[\CO_{\bQG}]=[\CO_{\bQG}]\cdot\varrho_0\Big(\sum_{i=1}^{n+1}X^{-\varepsilon_i}\Big)_{e,e}
\end{align}
and we can use \eqref{E:K-w-arb} to compute the right-hand side. One can easily show that just a single term contributes to the $(e,e)$-entry for each $i=1,\dotsc,n+1$. The contributing terms are those given by the stationary walk $(e,\dotsc,e)\in \QWalk^{(i-1)}_e$. Thus \eqref{E:K-w-arb} immediately results in the following:
\begin{align}\label{E:q-toda}
\sum_{i=1}^{n+1}e^{\varepsilon_i}\cdot[\CO_{\bQG}]&=[\CO_{\bQG}]\cdot\Big(X^{-w_\circ\varepsilon_1}+\sum_{i=2}^{n+1}X^{-w_\circ\varepsilon_i}(1-q\cdot t_{-w_\circ\alpha_{i-1}} X^{-w_\circ\alpha_{i-1}})\Big).
\end{align}
The element of $\heis$ acting on the right-hand side of \eqref{E:q-toda} is an equivalent form of the usual first order $q$-Toda difference operator, cf. \cite[(2)]{GL03}.

\section{Proof of Proposition~\ref{P:L}.}\label{A:L}

Recall our setting from Section~\ref{S:main-limit}. 

\subsection{Proof of \eqref{E:l-x-lng}.}

We know that 
\begin{equation*}
\rho-x^{-1}\rho = \sum_{\alpha \in \Inv (x^{-1})} \alpha
\end{equation*}
Because $\Inv (x^{-1}) \subset \DJp$, and 
because $\vpi_{k}$ is minuscule, we see that 
$(\alpha,\vpi_{k})=1$ for all $\alpha \in \Inv (x^{-1})$. 
Thus, 
\begin{equation*}
(\rho-x^{-1}\rho,\vpi_{k}) = 
\sum_{\alpha \in \Inv (x^{-1})} (\alpha,\vpi_{k})= \ell(x^{-1})=\ell(x),
\end{equation*}
and hence
$(\rho,\vpi_{k})-(\rho,\lambda) = \ell(x)$. 
Similarly, we have 
$(\rho,\vpi_{k})-(\rho,\mcr{\lng}\vpi_{k}) = \ell(\mcr{\lng})$. 
Because $(\rho,\mcr{\lng}\vpi_{k}) = (\rho,\lng\vpi_{k}) = 
- (\rho,\vpi_{k})$, we get $\ell(\mcr{\lng}) = 2(\rho,\vpi_{k})$. 
Therefore we obtain
\begin{equation*}
2\ell(x) - \ell(\mcr{\lng}) = 
2(\rho,\vpi_{k})-2(\rho,\lambda) - 2(\rho,\vpi_{k}) = -2(\rho,\lambda), 
\end{equation*}
as desired. 

\subsection{Proof of \eqref{E:rho-w-rho}.}

We know that $\rho-w\rho = \sum_{\gamma \in \Inv(w)} \gamma$.
Also, we note that $(\lambda,\gamma) \in \{0,\pm 1\}$ 
for all $\gamma \in \Delta$ since $\lambda$ is minuscule.  
Thus we obtain
\begin{equation*}
(\rho-w\rho,\lambda) = 
\sum_{\gamma \in \Inv(w)} 
\underbrace{ (\gamma,\lambda) }_{\in \{0,\pm 1\}} = 
\#\Inv(w)_{\lambda}^{+} - \#\Inv(w)_{\lambda}^{-} = 
\ell_{\lambda}^{+}(w) - \ell_{\lambda}^{-}(w),
\end{equation*}
as desired. 

\subsection{Proofs of \eqref{E:inv-la-} and \eqref{E:l-diff-}.}
%
%
\begin{lem} \label{lema}
For $\alpha,\beta \in \Delta$, we have 
$(\alpha,\beta) \in \bigl\{ 0, \pm 1, \pm 2\bigr\}$. 
Also, $(\alpha,\beta) = \pm 2$ if and only if $\alpha=\pm \beta$. 
\end{lem}

\begin{proof}
Let $\alpha,\beta \in \Delta$ be such that 
$\alpha \ne \pm \beta$, and set
$a:=(\alpha,\beta)$. 
Suppose for a contradiction that $a \le -2$. 
Because $s_{\beta}(\alpha) = \alpha - a \beta \in \Delta$, 
it follows from the property of the $\beta$-string 
through $\alpha$ that $\alpha+\beta \in \Delta$. 
However, $(\alpha+\beta,\alpha+\beta) = (\alpha,\alpha)+2(\alpha,\beta)+
(\beta,\beta) = 4+2a \le 0$, which is a contradiction. 
Suppose for a contradiction that $a \ge 2$. If we set $\gamma:=-\beta$, 
then $\alpha \ne \pm \gamma$, and $(\alpha,\gamma) = -a \le -2$, 
which is a contradiction. Therefore we conclude that 
$(\alpha,\beta) \in \bigl\{0,\pm 1,\pm 2\bigr\}$ for all 
$\alpha,\beta \in \Delta$ such that $\alpha \ne \pm \beta$, as desired. 
\end{proof}

%
\begin{lem} \label{lemb}
\mbox{}
\begin{enu}
\item It holds that $(\lambda,\beta_{r}) = -1$ for all $1 \le r \le l$. 

\item It holds that $(\beta_{r},\beta_{t}) \in \bigl\{0,1\bigr\}$ 
for all $1 \le r,t \le l$ with $r \ne t$. 
\end{enu}
\end{lem}

\begin{proof}
(1) We see that 
\begin{align*}
(\lambda,\beta_{r}) & = 
(s_{j_{l}}s_{j_{l-1}} \cdots s_{j_{2}}s_{j_{1}}\vpi_{k},\,
 s_{j_{l}}s_{j_{l-1}} \cdots s_{j_{r+1}}(\alpha_{j_{r}}) ) \\
& = 
(\underbrace{ s_{j_{r}}s_{j_{r-1}} \cdots s_{j_{2}}s_{j_{1}} }_{=:v}
 \vpi_{k},\,\alpha_{j_{r}}) = (\vpi_{k},v^{-1}\alpha_{j_{r}}). 
\end{align*}
Since $v=s_{j_{r}}s_{j_{r-1}} \cdots s_{j_{2}}s_{j_{1}} \in \WJ$ 
with $s_{j_{r}}v  < v$, it follows that 
$v^{-1}\alpha_{j_{r}} \in -(\DJp)$. Hence, 
$(\vpi_{k},v^{-1}\alpha_{j_{r}}) < 0$; since $\vpi_{k}$ is minuscule, 
we obtain $(\vpi_{k},v^{-1}\alpha_{j_{r}}) = -1$, as desired. 

(2) By Lemma~\ref{lema}, we have 
$(\beta_{r},\beta_{t}) \in \bigl\{0,\pm 1\bigr\}$. 
Suppose, for a contradiction, that 
$(\beta_{r},\beta_{t}) = -1$. Then, $s_{\beta_{r}}(\beta_{t}) = 
\beta_{r}+\beta_{t} \in \Delta$. By (1), we have 
$(\lambda,\beta_{r}+\beta_{t}) = -2$, which contradicts 
the assumption that $\vpi_{k}$ is minuscule. 
Thus we have proved the lemma. 
\end{proof}

For $\alpha = \sum_{i \in I}c_{i}\alpha_{i} \in Q$, 
we set $\Ht(\alpha):=\sum_{i \in I} c_{i} \in \BZ$. 
%
%
\begin{lem} \label{lem2}
Let $1 \le r < t \le l$. 
If $(\beta_{r},\beta_{t})=1$, 
then $\Ht(\beta_{r}) > \Ht(\beta_{t})$. 
\end{lem}

\begin{proof}
We show the assertion by induction on $l=\ell(x)$. 
If $l=0$ or $l=1$, then the assertion is obvious. 
Assume that $l > 1$. If $t = l$, then 
$\beta_{t} = \beta_{l} = \alpha_{j_{l}}$, and 
\begin{equation*}
s_{j_{l-1}} \cdots s_{j_{r+1}}(\alpha_{j_{r}}) = 
s_{j_{l}} \beta_{r} = s_{\beta_{t}} \beta_{r} = 
\beta_{r}-\beta_{t}.
\end{equation*}
Thus, 
$\beta_{r} = \beta_{t} + s_{j_{l-1}} \cdots s_{j_{r+1}}(\alpha_{j_{r}})$; 
note that $s_{j_{l-1}} \cdots s_{j_{r+1}}(\alpha_{j_{r}}) \in \Delta^{+}$ 
since $s_{j_{l-1}} \cdots s_{j_{r+1}}s_{j_{r}}$ is reduced. 
Therefore we obtain $\Ht(\beta_{r}) > \Ht(\beta_{t})$.
 
Assume that $t \le l-1$. Notice that 
$x':=s_{j_{l}}x \in \WJ$ with $\ell(x') = \ell(x)-1$, and 
$x'=s_{j_{p-1}} \cdots s_{j_{2}}s_{j_{1}}$ is 
a reduced expression of $x'$. Define 
\begin{equation*}
\beta_{u}':=s_{j_{l-1}}
\cdots s_{j_{u+1}}(\alpha_{j_{u}}) = s_{j_{l}}\beta_{u} \qquad 
\text{for $1 \le u \le l-1$}. 
\end{equation*}
Because $1 \le r < t \le l-1$ satisfy the condition that 
$(\beta_{r}',\beta_{t}') = (\beta_{r},\beta_{t})=1$, 
it follows by the induction hypothesis that 
$\Ht(\beta_{r}') > \Ht(\beta_{t}')$. 
Here we remark that 
\begin{equation*}
\beta_{r} = \beta_{r}' - 
\underbrace{(\beta_{r}',\alpha_{j_{l}})}_{=:a} \alpha_{j_{l}}, \qquad 
\beta_{t} = \beta_{t}' - 
\underbrace{(\beta_{t}',\alpha_{j_{l}})}_{=:b} \alpha_{j_{l}},
\end{equation*}
where $a,\,b \in \{0,-1\}$ by Lemma~\ref{lemb}\,(2). 
If $a=-1$, or if $a=b=0$, then it is obvious that 
$\Ht(\beta_{r}) > \Ht(\beta_{t})$ 
since $\Ht(\beta'_{r}) > \Ht(\beta'_{t})$ 
by the induction hypothesis. 
Assume that $a=0$ and $b=-1$. 
Suppose, for a contradiction, that 
$\Ht(\beta_{r}) \le \Ht(\beta_{t})$. 
Since $\Ht(\beta_{t}) = \Ht(\beta'_{t}) + 1$ and 
$\Ht(\beta_{r}) = \Ht(\beta'_{r})$, and since 
$\Ht(\beta'_{r}) > \Ht(\beta'_{t})$ by the induction hypothesis, 
we have 
$\Ht(\beta_{r}) = \Ht(\beta_{t})$. 
Since $(\beta_{r},\beta_{t})=1$ by assumption, 
we see that $\beta_{r}-\beta_{t} \in \Delta$. 
However, $\Ht(\beta_{r}-\beta_{t}) =\Ht(\beta_{r})-\Ht(\beta_{t})=0$, 
which is a contradiction. Thus we get 
$\Ht(\beta_{r}) >  \Ht(\beta_{t})$, as desired. 
%
%
Thus we have proved Lemma~\ref{lem2}. 
\end{proof}
%
%
\begin{lem} \label{lem3}
Let $1 \le r \le l$, and set $h:=\Ht(\beta_{r})$. Then, 
\begin{align}
& \# \bigl\{ r < t \le l \mid (\beta_{t},\beta_{r})=1\bigr\}=h-1, 
  \label{eq3-1} \\
& \# \bigl\{ \alpha \in \Delta^{+} \mid 
\Ht(\alpha) < \Ht(\beta_{r}),\,
(\alpha,\beta_{r})=1 \bigr\} = 2(h-1). \label{eq3-2}
\end{align}
\end{lem}

\begin{proof}
We show the assertion by induction on $l=\ell(x)$. 
If $l=0$ or $l=1$, then the assertion is obvious. 
Assume that $l > 1$. If $r = l$, then the assertion is obvious. 
Assume that $1 \le r \le l-1$. Notice that 
$x':=s_{j_{l}}x \in \WJ$ 
with $\ell(x') = \ell(x)-1$, and 
$x'=s_{j_{l-1}} \cdots s_{j_{2}}s_{j_{1}}$ is 
a reduced expression of $x'$. Define 
\begin{equation*}
\beta_{u}':=s_{j_{l-1}}
\cdots s_{j_{u+1}}(\alpha_{j_{u}}) = s_{j_{l}}\beta_{u} \qquad 
\text{for $1 \le u \le l-1$}. 
\end{equation*}
Since $\beta_{l}=\alpha_{j_{l}}$, 
we see by Lemma~\ref{lemb}\,(2) that 
$(\beta_{r}', \alpha_{j_{l}}) = - (\beta_{r},\alpha_{j_{l}}) = 
- (\beta_{r},\beta_{l}) \in \{0,-1\}$. 

\paragraph{Case 1.}
%
Assume first that $(\beta_{r}', \alpha_{j_{l}}) = 0$, 
or equivalently, $(\beta_{r}, \alpha_{j_{l}}) = 0$; 
in this case, $\beta_{r}=\beta_{r}'$, and hence 
$\Ht(\beta_{r}')=\Ht(\beta_{r})=h$. 
Also, since $\beta_{l}=\alpha_{j_{l}}$, it follows that 
\begin{align*}
\bigl\{ r < t \le l \mid (\beta_{t},\beta_{r})=1\bigr\} & = 
\bigl\{ r < t \le l-1 \mid (\beta_{t},\beta_{r})=1\bigr\} \\
& = \bigl\{ r < t \le l-1 \mid 
(s_{j_{l}}\beta_{t},s_{j_{l}}\beta_{r})=1\bigr\} \\
& = \bigl\{ r < t \le l-1 \mid 
(\beta_{t}',\beta_{r}')=1\bigr\}. 
\end{align*}
By the induction hypothesis, 
we have $\# \bigl\{ r < t \le l-1 \mid 
(\beta_{t}',\beta_{r}')=1\bigr\} = h-1$, and hence 
$\# \bigl\{ r < t \le l \mid (\beta_{t},\beta_{r})=1\bigr\} =h-1$. 
Moreover, since $\beta_{r}=\beta_{r}'$ in Case 1, 
it is obvious that 
\begin{align*}
& 
\bigl\{ \alpha \in \Delta^{+} \mid 
\Ht(\alpha) < \Ht(\beta_{r}),\,
(\alpha,\beta_{r})=1\bigr\} \\
& = 
\bigl\{ \alpha \in \Delta^{+} \mid 
\Ht(\alpha) < \Ht(\beta'_{r}),\,
(\alpha,\beta'_{r})=1\bigr\}. 
\end{align*}
Since $\# \bigl\{ \alpha \in \Delta^{+} \mid 
\Ht(\alpha) < \Ht(\beta'_{r}),\,
(\alpha,\beta'_{r})=1\bigr\}= 2(h-1)$ by the induction hypothesis, 
we obtain \eqref{eq3-2}, as desired. 

\paragraph{Case 2.}
%
Assume next that 
$(\beta'_{r}, \alpha_{j_{l}}) = -1$, or equivalently, 
$(\beta_{r}, \alpha_{j_{l}}) = 1$;  
in this case, $\beta_{r}=\beta'_{r}+\alpha_{j_{l}}$, and hence 
$\Ht(\beta'_{r})=\Ht(\beta_{r})-1=h-1$. 
Also, since $\beta_{l}=\alpha_{j_{l}}$, it follows that 
\begin{align*}
\bigl\{ r < t \le l \mid (\beta_{t},\beta_{r})=1\bigr\} & = 
\bigl\{ r < t \le l-1 \mid (\beta_{t},\beta_{r})=1\bigr\} \cup \{l\} \\
& = \bigl\{ r < t \le l-1 \mid 
(s_{j_{l}}\beta_{t},s_{j_{l}}\beta_{r})=1\bigr\} \cup \{l\} \\
& = \bigl\{ r < t \le l-1 \mid 
(\beta_{t}',\beta_{r}')=1\bigr\} \cup \{l\}. 
\end{align*}
By the induction hypothesis, we have 
$\# \bigl\{ r < t \le l-1 \mid (\beta_{t}',\beta_{r}')=1\bigr\} = h-2$, 
and hence 
$\# \bigl\{ r < t \le l \mid (\beta_{t},\beta_{r})=1\bigr\} = h-1$. 

Now, let us prove \eqref{eq3-2} in Case 2. 
For simplicity of notation, we set
\begin{equation} \label{eq:RS}
\begin{split}
& R=R_{r}:=
\bigl\{ \alpha \in \Delta^{+} \mid 
\Ht(\alpha) < \Ht(\beta_{r}),\,
(\alpha,\beta_{r})=1\bigr\}, \\
& S=S_{r}:=
\bigl\{ \alpha \in \Delta^{+} \mid 
\Ht(\alpha) < \Ht(\beta'_{r}),\,
(\alpha,\beta'_{r})=1\bigr\}.
\end{split}
\end{equation}
Because $(\beta_{r}, \alpha_{j_{l}}) = 1$, and 
$\beta_{r}-\alpha_{j_{l}} = 
s_{j_{l}}\beta_{r} = \beta'_{r} \in \Delta^{+}$,  
we deduce that 
$\alpha_{j_{l}},\,\beta_{r}-\alpha_{j_{l}} \in R$. 
We claim that 
%
%
\begin{equation} \label{eq:Sa}
s_{j_{l}}\alpha \in S \qquad 
\text{for all $\alpha \in R \setminus 
\bigl\{ \alpha_{j_{l}},\, \beta_{r}-\alpha_{j_{l}} \bigr\}$}. 
\end{equation}
It can be easily checked that 
$s_{j_{l}}\alpha \in \Delta^{+}$ and 
$(s_{j_{l}}\alpha,\beta'_{r}) = 1$. 
We show that $\Ht(s_{j_{l}} \alpha) < \Ht(\beta'_{r})$. 
By Lemma~\ref{lema}, 
$(\alpha,\alpha_{j_{l}}) \in \{0,\pm 1\}$.
If $(\alpha,\alpha_{j_{l}}) = 1$, then 
it is obvious that $\Ht(s_{j_{l}} \alpha) < \Ht(\beta'_{r})$. 
Assume that $(\alpha,\alpha_{j_{l}}) = 0$. 
Suppose, for a contradiction, that 
$\Ht(s_{j_{l}} \alpha) \ge \Ht(\beta'_{r})$. 
Since $\Ht(\alpha) < \Ht(\beta_{r})$ and 
$\Ht(\beta'_{r}) = \Ht(\beta_{r}) -1$, we have 
$\Ht(s_{j_{l}} \alpha) = \Ht(\beta'_{r})$. 
Since $(s_{j_{l}} \alpha,\beta'_{r})=1$, 
we see that $\beta'_{r} - s_{j_{l}} \alpha \in \Delta$. 
However, $\Ht(\beta'_{r} - s_{j_{l}} \alpha) =0$, 
which is a contradiction. 
Assume that 
$(\alpha,\alpha_{j_{l}}) = -1$; in this case, 
$s_{j_{l}}\alpha = \alpha+\alpha_{j_{l}}$. 
Suppose, for a contradiction, that 
$\Ht(s_{j_{l}}\alpha) \ge \Ht(\beta'_{r})$; 
since $\Ht(s_{j_{l}}\alpha) = \Ht(\alpha)+1$ and 
$\Ht(\beta'_{r}) = \Ht(\beta_{r})-1$, and since 
$\Ht(\alpha) < \Ht(\beta_{r})$, it follows that 
$\Ht(s_{j_{l}}\alpha)$ is equal to either 
$\Ht(\beta'_{r})$ or $\Ht(\beta'_{r})+1$. 
By the same reasoning as above, 
we see that $\Ht(s_{j_{l}}\alpha) \ne \Ht(\beta'_{r})$. 
Hence, $\Ht(s_{j_{l}}\alpha) = \Ht(\beta'_{r})+1$, 
and hence $\Ht(\alpha)=\Ht(\beta_{r})-1$. 
Also, we see that $s_{\alpha}\beta_{r} = \beta_{r}-\alpha \in \Delta^{+}$. 
By these facts, we deduce that $\alpha = \beta_{r}-\alpha_{j}$ 
for some $j \in I$. Then,
\begin{equation*}
-1 = (\alpha,\alpha_{j_{l}}) = 
(\beta_{r},\alpha_{j_{l}})-
(\alpha_{j},\alpha_{j_{l}}) = 
1 - (\alpha_{j},\alpha_{j_{l}}), 
\end{equation*}
and hence $(\alpha_{j},\alpha_{j_{l}}) = 2$. 
Therefore we obtain $j = j_{l}$. Thus, 
$\alpha = \beta_{r}-\alpha_{j_{l}}$, 
which is a contradiction. 
Thus we have proved 
$\Ht(s_{j_{l}} \alpha) < \Ht(\beta'_{r})$ in all cases. 
Hence we obtain the map
\begin{equation*}
s_{j_{l}} : R \setminus 
\bigl\{ \alpha_{j_{l}},\, \beta_{r}-\alpha_{j_{l}} \bigr\} 
\rightarrow S, \qquad \alpha \mapsto s_{j_{l}}\alpha.
\end{equation*}
Similarly, we obtain the inverse map 
\begin{equation*}
s_{j_{l}} : S \rightarrow R \setminus 
\bigl\{ \alpha_{j_{l}},\, \beta_{r}-\alpha_{j_{l}} \bigr\}, 
\qquad \alpha \mapsto s_{j_{l}}\alpha.
\end{equation*}
Recall that $\Ht(\beta'_{r})=\Ht(\beta_{r})-1=h-1$. 
We get
\begin{align*}
\# R & = \# (R \setminus 
\bigl\{ \alpha_{j_{l}},\, \beta_{r}-\alpha_{j_{l}} \bigr\}) + 2 \\
& = \# S + 2 = 2(h-2) + 2 = 2(h-1),
\end{align*}
as desired. Thus we have proved Lemma~\ref{lem3}. 
\end{proof}
%
%
\begin{rem} \label{remb}
Keep the notation and setting above. 
Let $1 \le r \le l$. We set
\begin{equation}
R_{r}:=
\bigl\{ \alpha \in \Delta^{+} \mid 
\Ht(\alpha) < \Ht(\beta_{r}),\,
(\alpha,\beta_{r})=1\bigr\}. 
\end{equation}
If $\alpha \in R_{r}$, then 
$s_{\alpha}\beta_{r} = \beta_{r}-\alpha \in \Delta^{+}$; 
it can be easily checked that $\beta_{r}-\alpha \in R_{r}$. 
Thus, $\sigma:\alpha \mapsto \beta_{r}-\alpha$ 
is an involution on $R_{r}$; 
notice that $\sigma$ is fixed-point-free 
(otherwise, $\beta_{r}=2\alpha$ for some $\alpha \in R_{r}$). 
Now, by Lemma~\ref{lem2},
\begin{equation*}
B_{r}:=
\bigl\{ \beta_{t} \mid r < t \le l,\,(\beta_{t},\beta_{r})=1\bigr\} 
\end{equation*}
is a subset of $R_{r}$; recall from Lemma~\ref{lem3} 
that $\#B_{r}=h-1$ and $\#R_{r}=2(h-1)$, 
where $h:=\Ht(\beta_{r})$. We claim that 
$\sigma(B_{r}) \sqcup B_{r} = R_{r}$. Indeed, 
suppose, for a contradiction, that 
$\sigma(\beta_{t}) = \beta_{r}-\beta_{t} \in B_{r}$ 
for some $\beta_{t} \in B_{r}$. 
Let $r < s \le l$ be such that 
$\sigma(\beta_{t}) = \beta_{s}$. Then, 
$\beta_{r} = \beta_{s}+\beta_{t}$, 
and hence $(\lambda,\beta_{r}) = (\lambda, \beta_{s})+(\lambda,\beta_{t}) = -2$ by 
Lemma~\ref{lemb}\,(1). However, this contradicts the fact that 
$\lambda$ is minuscule. Thus, $\sigma(B_{r}) \cap B_{r} = \emptyset$, 
and hence $\# ( \sigma(B_{r}) \cup B_{r} ) = 2(h-1) = \# R_{r}$. 
Therefore, $\sigma(B_{r}) \sqcup B_{r} = R_{r}$. 
\end{rem}

\begin{proof}[Proofs of \eqref{E:inv-la-} and \eqref{E:l-diff-}.]
Equation \eqref{E:inv-la-} follows from Lemma~\ref{lemb}\,(1). 
Let us prove equation \eqref{E:l-diff-}.
For each $v \in \bigl\{w,\,s_{\beta_{r}}w\bigr\}$, we set
\begin{align*}
X(v) & := \bigl\{ \gamma \in \Inv(v) \mid 
s_{\beta_{r}}\gamma \in \Delta^{+} \bigr\}, \\
Y(v) & := \bigl\{ \gamma \in \Inv(v) \mid 
s_{\beta_{r}}\gamma \in \Delta^{-} \bigr\};
\end{align*}
note that 
$\Inv (v) = X(v) \sqcup Y(v)$
for each $v \in \bigl\{w,\,s_{\beta_{r}}w\bigr\}$. 
We see that the map $X(w) \rightarrow X(s_{\beta_{r}}w)$, 
$\gamma \mapsto s_{\beta_{r}}\gamma$, is bijective. 
Hence,
\begin{equation*}
\ell(w)-\ell(s_{\beta_{r}}w) = 
\# Y(w) - \# Y(s_{\beta_{r}}w). 
\end{equation*}
Also, by Lemma~\ref{lema} and the assumption that 
$s_{\beta_{r}}w < w$, it follows that 
\begin{equation*}
Y(w) = \{ \beta_{r} \} \sqcup 
(\Inv (w) \cap R_{r}), \qquad 
Y(s_{\beta_{r}} w) = 
\Inv(s_{\beta_{r}}w) \cap R_{r};
\end{equation*}
for the definition of $R_{r}$, see Remark~\ref{remb}. 
Also, recall from Remark~\ref{remb} that 
for $\beta_{t} \in B_{r}$, 
$\sigma(\beta_{t}) = \beta_{r}-\beta_{t} \in R_{r} \setminus B_{r}$, and 
that $R_{r}=B_{r} \sqcup \sigma(B_{r})$.  
Since $\beta_{r} \in \Inv(w)$, we deduce that 
for each $\beta_{t} \in B_{r}$, 
one of the following (i), (ii), and (iii) holds: 
\begin{enumerate}
\item[(i)] $\beta_{t},\,\sigma(\beta_{t}) \in \Inv(w)$; 

\item[(ii)] $\beta_{t} \in \Inv(w)$ and 
$\sigma(\beta_{t}) \not\in \Inv(w)$; 

\item[(iii)]
$\beta_{t} \not\in \Inv(w)$ and 
$\sigma(\beta_{t}) \in \Inv(w)$. 
\end{enumerate}
We see that (i), (ii), and (iii) are equivalent to 
the following (i)', (ii)', and (iii)', respectively: 
\begin{enumerate}
\item[(i)'] $\beta_{t},\,\sigma(\beta_{t}) \not\in \Inv(s_{\beta_{r}}w)$; 

\item[(ii)'] $\beta_{t} \in \Inv(s_{\beta_{r}}w)$ and 
$\sigma(\beta_{t}) \not\in \Inv(s_{\beta_{r}}w)$; 

\item[(iii)'] $\beta_{t} \not\in \Inv(s_{\beta_{r}}w)$ and 
$\sigma(\beta_{t}) \in \Inv(s_{\beta_{r}}w)$. 
\end{enumerate}
Hence, if we set
\begin{align*}
& a:=\# \bigl\{ \beta_{t} \in B_{r} \mid 
  \beta_{t},\,\sigma(\beta_{t}) \in \Inv(w) \bigr\}, \\
& b:=\# \bigl\{ \beta_{t} \in B_{r} \mid 
  \beta_{t} \in \Inv(w),\,
  \sigma(\beta_{t}) \not\in \Inv(w) \bigr\}, \\
& c:=\# \bigl\{ \beta_{t} \in B_{r} \mid 
  \beta_{t} \notin \Inv(w),\,
  \sigma(\beta_{t}) \in \Inv(w) \bigr\}, 
\end{align*}
then 
\begin{align*}
\# Y(w) & = 1 + \#(\Inv (w) \cap R_{r}) = 1 + 2a + b + c, \\
\# Y(s_{\beta_{r}}w) & = 
\#(\Inv (s_{\beta_{r}}w) \cap R_{r}) = b + c,
\end{align*}
and hence 
\begin{equation} \label{eq4a}
\ell(w)-\ell(s_{\beta_{r}}w) = 
\# Y(w) - \# Y(s_{\beta_{r}}w) = 1+2a. 
\end{equation}

Now, we compute 
\begin{align*}
& \ell_{\lambda,r}^{-}(w)-\ell_{\lambda,r}^{-}(s_{\beta_{r}}w) \\
& = \# \bigl( \Inv(w) \cap 
\bigl\{\beta_{r},\beta_{r+1},\dots,\beta_{l}\bigr\} \bigr) - 
\# \bigl( \Inv(s_{\beta_{r}}w) \cap 
\bigl\{\beta_{r},\beta_{r+1},\dots,\beta_{l}\bigr\} \bigr) \\
& = 
\# \bigl(\{ \beta_{r} \} \sqcup \bigl(\Inv(w) \cap 
\bigl\{\beta_{r+1},\dots,\beta_{l}\bigr\} \bigr)\bigr) - 
\# \bigl( \Inv(s_{\beta_{r}}w) \cap 
\bigl\{\beta_{r+1},\dots,\beta_{l}\bigr\} \bigr)=:(\ast),
\end{align*}
where the last equality follows from 
the assumption that $s_{\beta_{r}}w < w$. 
Also, we deduce that for $r < s \le l$ with 
$(\beta_{s}, \beta_{r})=0$, $\beta_{s} \in \Inv(w)$ 
if and only if $\beta_{s} \in \Inv(s_{\beta_{r}}w)$. 
Thus, 
\begin{align*}
(\ast) & = 1 + 
\# (\Inv(w) \cap B_{r}) - 
\# (\Inv(s_{\beta_{r}}w) \cap B_{r}) \\
& = 1 + (a+b) - b = 1+a.
\end{align*}
Therefore, we obtain 
\begin{equation*}
2 ( \ell_{\lambda,r}^{-}(w)-\ell_{\lambda,r}^{-}(s_{\beta_{r}}w) )  - 1
=2 ( 1+a ) - 1 = 2a+1 \stackrel{\eqref{eq4a}}{=} 
\ell(w)-\ell(s_{\beta_{r}}w),
\end{equation*}
as desired. Thus we have proved \eqref{E:l-diff-}. 
\end{proof}

\subsection{Proofs of \eqref{E:inv-la+} and \eqref{E:l-diff+}.}
%
Let $\omega:I \rightarrow I$ be the Dynkin diagram automorphism 
induced by the longest element $\lng$, that is, $\lng \alpha_{i}=
-\alpha_{\omega(i)}$ for $i \in I$; notice that 
$\mcr{\lng}\vpi_{k} = \lng \vpi_{k} = - \vpi_{\omega(k)}$. 
Thus we see that $\vpi_{\omega(k)}$ is also minuscule. Also, we deduce that 
$y^{-1}=s_{i_{m}} \cdots s_{i_{1}} \in W^{\omega(\J)} = W^{I \setminus \{\omega(k)\}}$. 
We set $\nu:=y^{-1}\vpi_{\omega(k)}$; for $\gamma \in \Delta^{+}$, 
\begin{equation*}
(\gamma,\lambda)=1 \iff (\gamma,y^{-1}\mcr{\lng}\vpi_{k})=1 \iff 
(\gamma,y^{-1}(-\vpi_{\omega(k)}))=1 \iff (\gamma,\nu) = -1.
\end{equation*}
Hence it follows that $\Inv(w)_{\lambda}^{+} = \Inv(w)_{\nu}^{-}$ 
for $w \in W$. Also, it can be easily checked that 
$\ell_{\lambda,s}^{+}(w) = \ell_{\nu,s}^{-}(w)$ for all $w \in W$. 
Thus equations \eqref{E:l-diff-} and \eqref{E:l-diff+} follow from 
equations \eqref{E:l-diff+} and \eqref{E:l-diff-} (applied to $\nu$), 
respectively. 

Thus we have proved Proposition~\ref{P:L}.

\end{document}